\documentclass[11pt]{amsart}
\usepackage{amsmath, amscd, amsthm, amssymb}
\usepackage{graphicx}

\usepackage{fullpage}

\def\C{{\mathbf C}}
\def\R{{\mathbf R}}
\def\Z{{\mathbf Z}}
\def\Q{{\mathbf Q}}

\def\blank{\phantom{x}}

\newtheorem{theorem}{Theorem}[section]
\newtheorem{thm}[theorem]{Theorem}
\newtheorem{lemma}[theorem]{Lemma}

\newtheorem{proposition}[theorem]{Proposition}
\newtheorem{prop}[theorem]{Proposition}

\newtheorem{claim}[theorem]{Claim}

\theoremstyle{definition}
\newtheorem{definition}[theorem]{Definition}
\newtheorem{defin}[theorem]{Definition}

\theoremstyle{remark}
\newtheorem{remark}[theorem]{Remark}

\DeclareMathOperator{\charf}{char}
\renewcommand{\(}{\left(}
\renewcommand{\)}{\right)}
\newcommand{\set}[1]{\left\lbrace#1\right\rbrace}

\newcommand{\mm}[4]{\(\begin{smallmatrix} #1 & #2\\ #3 & #4\end{smallmatrix}\)}
\newcommand{\mb}[4]{\(\begin{array}{cc} #1 & #2\\ #3 & #4\end{array}\)}

\DeclareMathOperator{\diag}{diag}
\DeclareMathOperator{\tr}{tr}
\DeclareMathOperator{\ind}{Ind}

\DeclareMathOperator{\Spin}{Spin}

\DeclareMathOperator{\GSp}{GSp}
\DeclareMathOperator{\PGSp}{PGSp}

\DeclareMathOperator{\GL}{GL}

\begin{document}
\title[Spin on $\GSp_6$]{The Spin $L$-function on $\GSp_6$ via a non-unique model}
\author{Aaron Pollack}
\address{Department of Mathematics, Stanford University, Stanford, CA 94305, USA}
\email{aaronjp@stanford.edu}
\author{Shrenik Shah}
\address{Department of Mathematics, Columbia University, New York, NY 10027, USA}
\email{snshah@math.columbia.edu}

\thanks{A.P.\ has been partially supported by NSF grant DMS-1401858.  S.S.\ has been supported in parts through NSF grants DGE-1148900, DMS-1401967, and also through the Department of Defense (DoD) National Defense Science \& Engineering Graduate Fellowship (NDSEG) Program.}

\begin{abstract}
We give two global integrals that unfold to a non-unique model and represent the partial Spin $L$-function on $\PGSp_6$.  We deduce that for a wide class of cuspidal automorphic representations $\pi$, the partial Spin $L$-function is holomorphic except for a possible simple pole at $s=1$, and that the presence of such a pole indicates that $\pi$ is an exceptional theta lift from $\mathrm{G}_2$.  These results utilize and extend previous work of Gan and Gurevich, who introduced one of the global integrals and proved these facts for a special subclass of these $\pi$ upon which the aforementioned model becomes unique.  The other integral can be regarded as a higher rank analogue of the integral of Kohnen-Skoruppa on $\GSp_4$.
\end{abstract}

\maketitle

\section{Introduction}
We give new Rankin-Selberg integral representations for the Spin $L$-function of cuspidal automorphic representations $\pi$ on $\PGSp_6$.  Previously, Bump-Ginzburg~\cite{bg} and Vo~\cite{vo} had considered the case when $\pi$ is generic.  The integral representations in this paper do not assume genericity.  Instead, the global integrals we write down unfold to a Fourier coefficient that does not appear to factorize for general cusp forms.  Specifically, we use a Fourier coefficient attached to the unipotent class $(4\ 2)$, in the notation of \cite{grs}.  It follows from \cite[Theorem 2.7]{grs} that the wavefront set of every cuspidal automorphic representation on $\GSp_6$ contains a coefficient of type $(6)$, $(4\ 2)$, or $(2\ 2\ 2)$, so having an integral representation that unfolds to the $(4\ 2)$ coefficient is very desirable.  (In this notation, the representations supporting the coefficient (6) are precisely the generic ones.)  We will analyze the integrals using a method that slightly generalizes the approach of Piatetski-Shapiro and Rallis \cite{psr}.  By combining our calculation with results of Langlands \cite{langlandsEP}, Shahidi \cite{shahidi}, and works of Gan and Gurevich \cite{gan,gg}, we can deduce consequences for possible poles of the Spin $L$-function and a relation to the image of the exceptional theta lifting from the split group of type $\mathrm{G}_2$.  See \cite{grs2} and \cite{gj} for more on this theta lift.

The method of \cite{psr} was developed to handle integral representations that unfold to a model that is not unique, such as the classical Rankin-Selberg integral representation discovered by Andrianov \cite{a1}.  Bump-Furusawa-Ginzburg \cite{bfg} use this method to analyze several integral representations for $L$-functions on $\GL_n$ and classical groups that unfold to a non-unique model.  More recently, Gurevich and Segal \cite{gs} produce a Rankin-Selberg integral representation unfolding to a non-unique model for the degree seven $L$-function on the exceptional group $\mathrm{G}_2$, and our paper \cite{pollackShahKS} puts the Rankin-Selberg integral of Kohnen and Skoruppa on $\GSp_4$ in the context of a non-unique model.

\subsection{The global integrals} We give two global integrals that inherit their meromorphicity from an Eisenstein series on $\GL_2$.  The first integral we give has been considered previously by Gan-Gurevich~\cite{gg}, but when restricted to a special class of CAP forms on $\GSp_6$ constructed by Ginzburg \cite{ginzburg}.  For these forms, Gan-Gurevich~\cite{gg} show that the Fourier coefficient alluded to above factorizes, and hence that the global integral is an Euler product.

To describe the integrals more precisely, first fix an auxiliary \'etale quadratic extension $L = \mathbf{Q}(\sqrt{D})$ of $\mathbf{Q}$.  Denote by $\GL_{2,L}^*$ the algebraic group of elements of $\GL_{2,L}$ with determinant in $\mathbf{Q}$.  Set $H$ the subgroup of elements $(g_1, g_2)$ in $\GL_2 \times \GL_{2,L}^*$ with $\det g_1 = \det g_2$.  There is an embedding $H \rightarrow \GSp_6$.  Set $E^*(g_1,s)$ the normalized Eisenstein series on $\GL_2$, and let $\phi$ be a cusp form in the space of $\pi$, which is assumed to have trivial central character.  We set
\[ I(\phi,s) = \int_{H(\Q)Z(\mathbf A)\backslash H(\mathbf A)} E^*(g_1,s)\phi(g) dg.\] 

As is shown in Gan-Gurevich~\cite{gg}, the integral $I(\phi,s)$ unfolds.  Denote by $R$ the standard parabolic subgroup of $\GSp_6$ with Levi $\GL_1 \times \GL_1 \times \GL_2$, and by $U_R$ its unipotent radical.  Associated to the quadratic extension $L$ is a nondegenerate unitary character 
\[\chi: U_R(\mathbf Q)\backslash U_R(\mathbf A) \rightarrow \mathbf C^\times.\]
Define the Fourier coefficient
\begin{equation}\label{phiChi}\phi_\chi(g) = \int_{U_R(\Q)\backslash U_R(\mathbf A)} \chi^{-1}(u)\phi(ug) du.\end{equation}
Then one has
\[I(\phi,s) = \int_{N(\mathbf A)Z(\mathbf A)\backslash H(\mathbf A)} f^*(g_1,s)\phi_{\chi}(g) dg,\]
where $f^*(g_1,s)$ is the (normalized) section defining the Eisenstein series, and $N$ is the unipotent radical of a Borel subgroup of $H$.

Recall that a $(U_R,\chi)$-model on an irreducible admissible representation $(\pi_p,V)$ is a linear functional $\Lambda: V \rightarrow \C$ satisfying $\Lambda(\pi(u)v) = \chi(u)\Lambda(v)$ for all $u \in U_R, v \in V$.  The following is our first main result.
\begin{theorem} \label{thm:introunramcalc}
Let $\pi = \otimes \pi_v$ be a cuspidal automorphic representation of $\PGSp_6$.  Suppose that $p$ is finite, $p\nmid 2D$, and $\pi_p$ is unramified.  Let $\Lambda$ be any $(U_R,\chi)$-model on $\pi_p$ and let $v_0$ be a spherical vector.  Set $\lambda(g) = \Lambda(\pi(g)v_0)$ and $f_p^*(g_1,s)$ the normalized spherical section in $Ind_{B(\Q_p)}^{\GL_2(\Q_p)}(\delta_B^s)$ (see Section \ref{subsec:unfolding}).  Then
\[\int_{N(\mathbf{Q}_p)Z(\mathbf{Q}_p)\backslash H(\mathbf{Q}_p)} f_p^*(g_1,s)\lambda(g) dg\]
is $\lambda(1)$ times the local Spin $L$-function, $L(\pi_p, \Spin,s)$. 
\end{theorem}
This theorem together with the ramified computation in Section \ref{spin:Ram} shows that one may choose the data in the global integral so that it is equal to the product of $L^S(\pi,\Spin,s)$ and an Archimedean integral $I_\infty(\phi,s)$.

For the second integral, we use the special automorphic function $P_D^\alpha$ on $\GSp_4$ (see Section \ref{spin:global}), defined in \cite{pollackShahKS}.  Denote by $\GL_2 \boxtimes \GSp_4$ the set of elements $(g_1,g_2)$ in $\GL_2 \times \GSp_4$ with $\det g_1 = \nu(g_2)$, where $\nu$ is the similitude.  Then $\GL_2 \boxtimes \GSp_4$ embeds in $G = \GSp_6$.  We consider the global integral
\[I'(\phi,s) = \int_{(\GL_2 \boxtimes \GSp_4)(\Q)Z(\mathbf A)\backslash (\GL_2 \boxtimes \GSp_4)(\mathbf A)} E^{*}(g_1,s)\phi(g) P_D^\alpha(g_2) dg,\]
for a cusp form $\phi$ on $\PGSp_6$. 

This integral also unfolds to the Fourier coefficient $\phi_\chi$, and we show that it represents the partial Spin $L$-function in the same sense as described for $I(\phi,s)$ above.  In fact, the unramified calculation, which is the main theorem of this paper, is the same for both of these integrals.
\begin{theorem} \label{thm:unramcalc}
Let $\lambda$ be as in Theorem \ref{thm:introunramcalc}, and denote by $\iota: \Q_p^\times \times L_p^\times \rightarrow \GSp_6(\Q_p)$ the map given in Definition \ref{def:iota}.  Then
\[\zeta_p(2s)\sum{|t|^{-3}|\ell|^2 \lambda(\iota(t,\ell))|t|^s} = \lambda(1)L(\pi_p,\Spin,s),\]
where the sum extends over all $t \in \mathbf{Q}_p^\times$ and $\ell \in L_p^\times$ satisfying $|t| \leq |N_{L_p/\Q_p}(\ell)| \leq 1$ modulo the action of $\mathbf{Z}_p^\times \times \mathcal O_L^\times$.
\end{theorem}

It is worth remarking that the integrals $I$ and $I'$ are closely related.  Indeed, integrating over $\GL_{2,L}^*(\mathbf Q)\backslash \GL_{2,L}^*(\mathbf A)$ and integrating against $P_D^\alpha(g)$ are very closely related functionals on the space of automorphic forms on $\GSp_4$.  Nevertheless, it is desirable to have both integrals written down explicitly.  The reason is that, on the one hand, the integral $I(\phi,s)$ is more obviously connected to periods of cusp forms and the geometry of the Siegel six-fold.  On the other hand, the use of the special function $P^\alpha_D(g)$ makes $I'(\phi,s)$ a more flexible construction.

\subsection{Overview of proof} To prove Theorem \ref{thm:unramcalc}, we use a result of Andrianov~\cite{a3}, who explicitly relates $L(\pi_p,\Spin,s)$ to Hecke operators.   Define $\Delta^s(g) = |\nu(g)|^s \charf(g)$, where $\nu$ is the similitude, $|\cdot|$ denotes the $p$-adic norm, and $\charf(g)$ denotes the characteristic function of $M_6(\mathbf Z_p)$.  Let $\omega(g)$ be the spherical function for $\pi_p$, normalized so that $\omega(1) = 1$.  Then Andrianov proves
\[\int_{\GSp_6(\mathbf Q_p)} \omega(g)\Delta^s(g) dg = N(\omega,s)L(\pi_p,\Spin,s-3)\]
where $N(\omega,s) = (N(s) * \omega)(1)$, and 
\begin{eqnarray*} N(s) &=&1 - p^2(T_{2,3} + (p^4 + p^2 + 1)T_{3,3})p^{-2s} + p^4(1+p)T_{0,3}T_{3,3}p^{-3s}\\
& &- p^7T_{3,3}(T_{2,3} + (p^4 + p^2 + 1)T_{3,3})p^{-4s}+p^{15}T_{3,3}^3p^{-6s}.
\end{eqnarray*}
Here the $T_{i,3}$ are the usual Hecke operators on $\GSp_6$. (See Section \ref{N}.)  Due to the bi-invariance of $\Delta$ by $\GSp_6(\mathbf{Z}_p)$, one has the identity
\[\int_{\GSp_6(\mathbf Q_p)} \lambda(g)\Delta^s(g) dg =\lambda(1)\int_{\GSp_6(\mathbf Q_p)} \omega(g)\Delta^s(g)dg.\]
Thus, to prove the theorem, it suffices to show that
\begin{equation} \label{eqn:maineq}\int_{\GSp_6(\mathbf Q_p)} \lambda(g)\Delta^s(g) dg = \zeta_p(2s-6)\sum{|t|^{-6}|\ell|^2 \Lambda(\pi(\iota(t,\ell))N(s)*v_0)|t|^s}.\end{equation}
We explicitly compute both sides and see that they are identical.

Note that in other integral representations involving non-unique models, there is an operator $N$ as above, but those $N$ act purely via the central character, and so are essentially trivial.  Here, the operator $N$ is highly nontrivial.  This is the sense in which our approach generalizes that of Piatetski-Shapiro and Rallis.  The paper \cite{gs} of Gurevich-Segal uses a similar strategy.

\subsection{Applications}
It is expected that the integral $I(\phi,s)$ will play an important role in answering arithmetic questions about the Siegel modular six-fold.  For instance, this is strongly suggested by the papers \cite{llz} and \cite{grossSavin}.  Although the motivic applications of the relation between $I(\phi,s)$ and the Spin $L$-function suggested by these papers remain currently out of reach, we may refine Theorem \ref{mainThm} to give more precise information about the poles of the partial Spin $L$-function of $\pi$ and their relation with the exceptional theta correspondence between $\mathrm{G}_2$ and $\PGSp_6$.  This extends the results of Gan-Gurevich \cite{gg}, and relies on their Archimedean calculation as well as earlier works of Langlands \cite{langlandsEP}, Shahidi \cite{shahidi}, and Gan \cite{gan}.

To set up this refinement, suppose that $\pi$ is a cuspidal automorphic representation on $\PGSp_6(\mathbf{A})$.  Denote by $U_P$ the unipotent radical of the Siegel parabolic of $\PGSp_6$, so that $U_P$ is abelian and consists exactly of the matrices $\mm{1}{X}{}{1}$ inside of $\PGSp_6$.  The characters $U_P(\Q)\backslash U_P(\mathbf{A}) \rightarrow \C^\times$ are indexed by $3\times 3$ symmetric matrices $T$ in $M_3(\Q)$; write $\chi_T$ for the character corresponding to $T$.  If $\phi$ is in the space of $\pi$, and $T$ is a $3\times 3$ symmetric matrix, one can consider the Fourier coefficient
\[ \phi_T(g) = \int_{U_P(\Q)\backslash U_P(\mathbf{A})}{\chi_T^{-1}(u)\phi(ug)\,du}.\]
One says that $\pi$ supports the $T$ Fourier coefficient if there exists $\phi$ in the space of $\pi$ so that $\phi_T$ is not identically zero.  We say $\pi$ supports a Fourier coefficient of \emph{rank $r$} if there exists $T$ of rank $r$ so that $\pi$ supports the $T$ Fourier coefficient.  For example, it is easy to verify that $\pi$ supports a rank one Fourier coefficient if and only if $\pi$ is globally generic.  Additionally, one checks that $\pi$ supports a rank two Fourier coefficient if and only if there is an etale quadratic extension $L$ of $\Q$ so that $\pi$ supports the coefficient (\ref{phiChi}).

If $\pi$ is globally generic, the analytic properties of the partial Spin $L$-function are well-understood, both from Rankin-Selberg integrals \cite{bg,vo} and from the Langlands-Shahidi method \cite{shahidi}.  If $\pi$ is not generic, then the method of Langlands \cite{langlandsEP} and Shahidi \cite[Theorem 6.1]{shahidi} still gives the meromorphic continuation of $L^S(\pi,\Spin,s)$, but does not provide other desired analytic properties of the $L$-function, such as the finiteness of poles or functional equation.  Combining Theorem \ref{mainThm} with the known meromorphic continuation of $L^S(\pi,\Spin,s)$, we deduce the meromorphic continuation of the Archimedean integral.  Proposition 12.1 of \cite{gg} then shows that for suitable data, the Archimedean integral $I_\infty(\phi,s)$ can be made nonvanishing at arbitrary $s=s_0$.  We deduce the finiteness of the poles of the partial Spin $L$-function.

Furthermore, Gan \cite[Theorem 1.2]{gan} and Gan-Gurevich \cite[Proposition 5.2]{gg} connect the periods of cusp forms in the space of $\pi$ over $H(\Q)Z(\mathbf{A})\backslash H(\mathbf{A})$ to the exceptional theta correspondence between $\mathrm{G}_2$ and $\PGSp_6$.  The period over this domain is calculated by the residue at $s=1$ of the Rankin-Selberg integral $I(\phi,s)$.  In particular, the combination of Theorem \ref{mainThm} below, Propositions 5.2 and 12.1 of Gan-Gurevich \cite{gg}, and Theorem 6.1 of Shahidi \cite{shahidi} yields the following result.

\begin{theorem}\label{simplePole} Suppose $\pi$ is a cuspidal automorphic representation of $\PGSp_{6/\mathbf{Q}}$ that supports a rank two Fourier coefficient.  Then the partial Spin $L$-function of $\pi$, $L^S(\pi,\Spin,s)$,  has meromorphic continuation in $s$, is holomorphic outside $s=1$, and has at worst a simple pole at $s=1$.  If $\mathrm{Res}_{s=1}L^S(\pi,\Spin,s) \neq 0$, then $\pi$ lifts to the split $\mathrm{G}_2$ under the exceptional theta correspondence. \end{theorem}

The reason we restrict to the ground field $\mathbf{Q}$ in Theorem \ref{simplePole} and the other theorems in this paper is because the result of Andrianov \cite{a3} that we use (and which is reproven in \cite{panchishkinVankov}) has only been verified for this ground field.  It would be a straightforward but tedious exercise to check that the results of \cite{a3} and \cite{panchishkinVankov}, and then Theorems \ref{thm:unramcalc} and \ref{simplePole}, extend to arbitrary global fields.

\subsection{Outline} The outline of the paper is as follows.  In Section \ref{spin:global} we define the various groups we use, recall some definitions and results pertaining to the special function $P_D^\alpha$, and give the unfolding of the global integrals.  In Section \ref{delta} we explicitly compute the left hand side of the equality (\ref{eqn:maineq}) and in Section \ref{N} we compute the right-hand side.  The ramified integral is controlled in Section \ref{spin:Ram}.

$\textrm{ }$

\noindent \emph{Acknowledgments}:  We would like to thank Christopher Skinner for many helpful conversations during the course of this research and for his constant encouragement.  We have also benefited from conversations with Wee Teck Gan, Nadya Gurevich, Erez Lapid, Peter Sarnak, and Xiaoheng Wang. 
\section{Global constructions}\label{spin:global}
\subsection{Groups and embeddings} \label{subsec:groupdefs}
For our symplectic form, we use the nonstandard matrix
\[\left(\begin{array}{cccc} & & &1\\ & & 1_2 & \\ & -1_2& & \\ -1& & & \end{array}\right).\]
We denote by $J_4:=\left(\begin{array}{cc} & 1_2 \\ -1_2 & \end{array}\right)$ the standard matrix defining the group $\GSp_4$.  Recall the group $\GL_2 \boxtimes \GSp_4$ from the introduction.  We embed $\GL_2 \boxtimes \GSp_4$ inside $\GSp_6$ via the map
\[\left(\left(\begin{array}{cc} a&b\\c&d\end{array}\right), \left(\begin{array}{cc} a'&b'\\c'&d'\end{array}\right)\right) \mapsto \left(\begin{array}{cccc} a& & &b\\ & a'& b'& \\ & c'&d' & \\ c& & &d \end{array}\right).\]
Here $a', b', c', d'$ are $2 \times 2$ matrices. If $W_1$ and $W_2$ are subgroups of $\GL_2$ and $\GSp_4$ respectively, we denote by $W_1 \boxtimes W_2$ the subgroup of elements $(w_1, w_2)$ in $\GL_2 \boxtimes \GSp_4$ with $w_1 \in W_1, w_2 \in W_2$.

With this choice of form, elements of the Levi of the parabolic $R$ from the introduction are of the form
\begin{equation}\left(\begin{array}{cccc} w& & & \\ &x & & \\ & &y& \\ & & &z \end{array}\right),\label{eqn:LeviR}\end{equation}
where $x$ and $y$ are $2 \times 2$ matrices and $x = wz{}^ty^{-1}$.  Elements of the unipotent radical $U_R$ of $R$ have the form
\[\left(
\begin{array}{cccc}
 1&   v &   r  & * \\
  & 1_2 &  u   & {}^tr -u{}^tv \\
  &     &  1_2 & -{}^tv \\
  &     &     &1 
\end{array}
\right)\]
with $u = {}^tu$ and no condition on the element $*$.

We will need several abelian subgroups of $U_R$. Define
\[N_V = \set{n_v \in \GSp_6 \left|n_v=\(\begin{array}{ccc|ccc} 1 & * & * &  &  &  \\ & 1 & &  &  &  \\  &  & 1 &  &  &  \\ \hline  &  &  & 1 &  & * \\  &  &  &  & 1 & * \\  &  &  &  &  & 1 \end{array}\right.\)},\]
\[N_U = \set{n_u \in \GSp_6 \left|n_u=\(\begin{array}{ccc|ccc} 1 & & &  &  &  \\ & 1 & & * & * &  \\  &  & 1 & * & * &  \\ \hline  &  &  & 1 &  &  \\  &  &  &  & 1 &  \\  &  &  &  &  & 1 \end{array}\right.\)},\]
\[\textrm{and }N_{U'} = \set{n_{u'} \in \GSp_6 \left|n_{u'}=\(\begin{array}{ccc|ccc} 1 & & & * & * & * \\ & 1 & &  &  & * \\  &  & 1 &  &  & * \\ \hline  &  &  & 1 &  &  \\  &  &  &  & 1 &  \\  &  &  &  &  & 1 \end{array}\right.\)}.\]
Typical elements of these groups are written $n_v, n_u,$ and $n_{u'}$, respectively.  If $n_v$ is an element of $N_V$, we write $v$ for the corresponding $1 \times 2$ matrix, and we write $V$ for the abelian group of these $1\times 2$ matrices.  Similarly, we write $u$ for the $2 \times 2$ symmetric matrix corresponding to an element $n_u$ of $N_U$, and we write $U$ for the group of these symmetric matrices. Sometimes we write $n(u)$ in place of $n_u$.  If $u =\mm{u_{11}}{u_{12}}{u_{12}}{u_{22}}$, $v = (v_1, v_2)$, and $n \in U_R$ with $n \equiv n_v n_u \pmod{[U_R,U_R]}$, we define $\chi (n)= \psi(v_1 - Du_{11}+u_{22})$. 

We describe the embedding of $\GL_{2,L}^*$ into $\GSp_4$.  Consider $(L^2, \langle \;,\; \rangle_L)$, the two-dimensional vector space over $L$ with its standard alternating bilinear form.  Define the $\mathbf{Q}$-bilinear form $\langle \;,\; \rangle$ via $\langle x,y\rangle = \tr_{L/\mathbf Q}\left(\langle x, y \rangle_L\right)$.  As $\GL_{2,L}^*$ is the group preserving $\langle \;,\;\rangle_L$ up to multiplication by $\mathbf Q^{\times}$, one obtains an inclusion $\GL_{2,L}^* \subseteq \GSp\left(\langle \; ,\; \rangle\right)$. If the basis of $L^2$ is $\{b_1,b_2\}$, then the basis
\[\frac{1}{2\sqrt{D}}b_1,\blank \frac{1}{2}b_1,\blank \sqrt{D}b_2, \blank b_2\]
gives an isomorphism $\GSp\left(\langle \; ,\; \rangle\right) \cong \GSp_4$.  We use this identification throughout.  Alternatively, with this choice of basis, scalar multiplication by $\sqrt{D}$ is given by the matrix
\[\left(\begin{array}{cc|cc} & 1& & \\ D& & & \\ \hline & & &D \\ & &1 & \end{array}\right),\]
and $\GL_{2,L}^*$ is the centralizer of this matrix inside $\GSp_4$.  Here and after we assume $\GSp_4$ acts on the right of its defining four-dimensional representation.  The group $H$ is embedded in $\GSp_6$ via the composite of the embeddings $\GL_{2,L}^* \subset \GSp_4$ and $\GL_2 \boxtimes \GSp_4 \subset \GSp_6$.

The Levi of the upper triangular Borel of $\GL_{2,L}^*$ is isomorphic to $\GL_{1} \times \GL_{1,L}$.  Suppose $\ell = x + y\sqrt{D} \in L^\times$ and $t \in \mathbf{Q}^\times$.  Then the embedding of $\GL_{2,L}^*$ into $\GSp_4$ sends
\begin{equation}\label{ML}\left(\begin{array}{c|c} t\ell^{-1} &  \\ \hline & \ell \end{array} \right)
\mapsto
\left(\begin{array}{c|c} t{}^tm_\ell^{-1} &  \\ \hline & m_\ell \end{array}\right),\end{equation}
with $m_\ell = \left(\begin{array}{cc} x & Dy \\ y & x \end{array} \right),$
and \[\left(\begin{array}{c|c} 1 & \ell \\ \hline  & 1 \end{array} \right)
\mapsto
\left(\begin{array}{c|c} 1 & n_\ell \\ \hline & 1 \end{array} \right),\]
with $n_\ell = \frac{1}{2}\left(\begin{array}{cc} x/D & y \\ y & x \end{array} \right).$

\begin{definition} \label{def:iota}

We define the map $\iota: \GL_{1} \times \GL_{1,L} \rightarrow \GSp_6$ via
\[(t,\ell) \mapsto
\left(
\begin{array}{cccc}
 t&    &       &  \\
  & t{}^tm_\ell^{-1} &     &  \\
  &     &  m_\ell &  \\
  &     &     &1 
\end{array}
\right).\]
Sometimes we will drop the $\iota$ from the notation and just write $(t, \ell)$ for $\iota(t,\ell)$ if no confusion is possible. Denote by $N_L$ the unipotent radical of the Borel of $\GL_{2,L}^*$.  Note that $\chi$ is trivial on the image of $N_L$ in $\GSp_6$ (embedded via $n \mapsto (1,n) \in H \subset \GSp_6$).

\end{definition}

We write $| \cdot |$ for the absolute value on $\Q_p$, normalized so that $|p| = p^{-1}$.  For an element $\ell$ of $L_p$, we also write $|\ell|$ for its absolute value, normalized so that $|\ell| = |N_{L_p/\Q_p}(\ell)|$, where $N_{L_p/\Q_p}(\ell)$ is the norm map from $L_p$ to $\Q_p$.  We will often write $N(\ell)$ for $N_{L_p/\mathbf{Q}_p}(\ell)$.  We record for future use some modulus characters.
\begin{lemma}\label{modulusChars} In the notation above, $\delta_R = |\nu|^6|z|^{-6}|\det(y)|^{-3}$, where $\nu$ is the similitude.  The modulus character for the Borel of $\GL_{2,L}^*$ is $|t|^3|\ell|^{-2}.$  If $P_4$ denotes the Siegel parabolic on $\GSp_4$, then $\delta_{P_4} = |\nu|^3|\det(y)|^{-3}$.
\end{lemma}

\subsection{Results on $P_D^\alpha$}
In this section we recall the definition of the function $P_D^\alpha$ on $\GSp_4$ from \cite{pollackShahKS} and a result related to it.

First, let $W_4$ be the symplectic space upon which $\GSp_4$ acts on the right and write $\wedge^2_0(W_4) = \ker\{\wedge^2 W_4 \rightarrow \nu\}$.  We set $V_5 = \wedge^2_0(W_4) \otimes \nu^{-1}$.  For a nonzero integer $D$ denote $v_D = De_1 \wedge f_2 + e_2 \wedge f_1$ in $V_5$.  If $\alpha$ is a Schwartz-Bruhat function on $V_5(\mathbf A)$, then we set
\[P_D^\alpha(g) = \sum_{\delta \in \GL_{2,L}^*(\Q)\backslash \GSp_4(\Q)}{\alpha(v_D\delta g)}.\]
This sum is well-defined because $\GL_{2,L}^*$ is the stabilizer of $v_D$ in $\GSp_4$.

When $p$ is split in $L$, we fix two uniformizers $\pi_1$, $\pi_2$ above $p$ such that $p= \pi_1\pi_2$, as follows.  Let $h$ be a fixed square root of $D$ inside $\mathbf Z_p$.  We define $\pi_1$ to be the element of $L_p^{\times}$ that maps to $p$ under the map $L_p\rightarrow \mathbf Q_p$ determined by $\sqrt{D}\mapsto h$ and maps to $1$ under the map determined by $\sqrt{D}\mapsto -h$.  The element $\pi_2$ is defined in the same way, with $h$ and $-h$ reversed. The following lemma will be useful in the sequel.
\begin{lemma}\label{lem:hlem} Suppose $p \nmid 2D$ is split in $L$.  If $\ell = \pi_1^k$, then $m_\ell$ is
\[\left(\begin{array}{cc} \frac{p^k+1}{2}&\frac{p^k-1}{2}h \\ \frac{p^k-1}{2}h^{-1}&\frac{p^k+1}{2} \end{array}\right),\]
and this matrix is right $\GL_2(\mathbf Z_p)$ equivalent to
\[\left(\begin{array}{cc} p^k&-h \\ &1 \end{array}\right).\]
Switching $h$ with $-h$ gives the analogous statement for $\pi_2$.
\end{lemma}
\begin{proof} For $\pi_1^k$, we must solve the equations $x + hy = p^k$ and $x-hy =1$ for $x$ and $y$, giving the first part.  The second part follows from a simple computation.  Switching $h$ and $-h$ gives the lemma for $\pi_2$.
\end{proof}

Let $U_{P,4}$ be the unipotent radical of the Siegel parabolic on $\GSp_4$, and by abuse of notation we write $\chi$ again for its restriction from $U_R$ to $U_{P,4}$.  Define
\[\alpha_\chi(g) = \int_{N_L(\mathbf A)\backslash U_{P,4}(\mathbf A)}\chi(u)\alpha(v_Dug)\,du = \prod_v{\alpha_\chi^v(g)}.\]
The following proposition computes $\alpha_\chi^v$ almost everywhere.
\begin{proposition}\label{alphachi} Denote by $T_L$ the torus of $\GL_{2,L}^*$, and $K_M$ the intersection of the Levi of the Siegel parabolic with $\GSp_4(\mathbf Z_p)$.  Suppose that $p$ does not divide $2D$, and that $\alpha^p$ is the characteristic function of $V_5(\mathbf Z_p)$.  If $m$ is in the Levi of the Siegel parabolic, then $\alpha_\chi^p(m) = 0$ if $m$ is not in $T_LK_M$.  If $m \in T_L$, then, in notation of (\ref{ML}), $\alpha_\chi^p(m) = |t||\ell|^{-1}$ when $|t| \leq |\ell|$ and $\alpha_\chi^p(m) = 0$ if $|t| > |\ell|$.
\end{proposition}
\begin{proof} This is \cite[Proposition 4.2]{pollackShahKS} combined with Lemma \ref{lem:hlem}.
\end{proof}

\subsection{Unfolding}\label{subsec:unfolding}
We define $E(g,s)$ by
\[E(g,s) = \sum_{\gamma \in B(\mathbf Q)\backslash \GL_2(\mathbf Q)}{f(\gamma g,s)},\]
where $f \in \ind(\delta_B^s)$, $\delta_B$ is the modulus character of the Borel, and $f = \prod{f_p}$ is factorizable.  (We are \emph{not} using normalized induction.)  At almost all places, $f_p$ is normalized so that $f_p(bk,s) = \delta_p(b)^s$ for $b$ upper triangular in $\GL_2(\mathbf Q_p)$ and $k \in \GL_2(\mathbf Z_p)$.  We also define a normalized Eisenstein series by $E^*(g,s) = \zeta(2s)E(g,s)$.  Set $f_p^*(g,s) = \zeta_p(2s)f_p(g,s)$.

\begin{proposition} The global integral $I(\phi,s)$ unfolds to 
\[\int_{N(\mathbf A)Z(\mathbf A)\backslash H(\mathbf A)} f^*(g_1,s)\phi_{\chi}(g) dg,\]
where $f^*(g_1,s) = \zeta(2s)f(g_1,s)$ is the (normalized) section defining the Eisenstein series, and $N$ is the unipotent radical of the Borel subgroup of $H$.  The global integral $I'(\phi,s)$ unfolds to
\[\int_{N_2(\mathbf A)U_{P,4}(\mathbf A)Z(\mathbf A)\backslash (\GL_2 \boxtimes \GSp_4)(\mathbf A)} f^{*}(g_1,s)\phi_{\chi}(g)\alpha_{\chi}(g_2) dg,\]
where $N_2$ is the unipotent radical of the Borel of $\GL_2$, $U_{P,4}$ is the unipotent radical of the Siegel parabolic on $\GSp_4$, and $f, \phi_{\chi}$, and $\alpha_{\chi}$ are as above.
\end{proposition}
\begin{proof} The integral $I(\phi,s)$ is unfolded in Gan-Gurevich~\cite{gg}. For the integral $I'(\phi,s)$, one first unfolds the sum defining $P_D^\alpha$.  Then one proceeds exactly as in the case of the integral $I(\phi,s)$.  \end{proof}

Even though the Fourier coefficient $\phi_\chi$ does not factorize, there is still a local integral corresponding to the unfolded global integrals $I(\phi,s)$ and $I'(\phi,s)$.  Namely, for $\Lambda$ a $(U_R,\chi)$-model for the representation $\pi_p$ and a vector $v_0$ in the space of $\pi_p$, we define
\[I(\Lambda,v_0,s) = \int_{N(\Q_p)Z(\Q_p)\backslash H(\Q_p)} f_p^*(g_1,s)\Lambda(\pi_p(g)v_0)dg.\]
Similarly, we define
\[I'(\Lambda,v_0,s)=\int_{N_2(\Q_p)U_{P,4}(\Q_p)Z(\Q_p)\backslash (\GL_2 \boxtimes \GSp_4)(\Q_p)} f_p^{*}(g_1,s)\Lambda(\pi_p(g)v_0)\alpha^p_{\chi}(g_2) dg.\]
When all the data in the integrals $I(\Lambda,v_0,s), I'(\Lambda,v_0,s)$ are unramified at some finite prime $p$, we will rewrite these local integrals as a sum.  Assume now that $v_0$ is a spherical vector and write $\lambda(g)=\Lambda(\pi_p(g)v_0)$.  We need the following condition restricting the support of $\lambda$.
\begin{lemma}\label{lem:lambdanonvan} If $g \in R$ has Levi part $m$ and $\lambda(g) \neq 0$, then in the notation ($\ref{eqn:LeviR}$), $z$ divides the top row of $y$.  If $\lambda(\iota(t,\ell)) \neq 0$, then $|t| \leq |\ell| \leq 1$.
\end{lemma}
\begin{proof} Take $n_v$ with $v \in V(\mathbf Z_p)$ and suppose $g = m n$ with $m \in M_R$ and $n \in U_R$.  Then
\[\lambda(g) = \lambda(mnn_v) = \lambda(mn'n_vn) = \chi(mn' n_v m^{-1})\lambda(mn) = \chi(mn_vm^{-1})\lambda(g),\]
where $n' \in [U_R,U_R]$.  Hence $\lambda(g) \neq 0$ implies $\chi(m n_v m^{-1})$ is $1$ for all $n_v$ with $v \in V(\mathbf Z_p)$.  It follows that $z$ divides the top row of $y$.

For the second claim, note that the part above gives $\ell \in \mathcal O_L$, i.e.\ $|\ell| \leq 1$.  Suppose $u \in U(\mathbf Z_p)$.  Then
\[\lambda(\iota(t,\ell)) = \lambda(\iota(t,\ell)n_u) = \chi(\iota(t,\ell)n_u \iota(t,\ell)^{-1})\lambda(\iota(t,\ell)).\]
One computes 
\[\iota(t,\ell)n_u \iota(t,\ell)^{-1} = n\left(\frac{t}{N(\ell)}u\right) + n'\]
with $n' \in N_L$.  Thus $\lambda(\iota(t,\ell)) \neq 0$ implies $|\frac{t}{N(\ell)}| \leq 1$.
\end{proof}
Maintain the assumption that $\pi_p$ is unramified and that $v_0$ is a spherical vector of $\pi_p$.  We also assume that $p \nmid 2D$, $\alpha^p$ is the characteristic function of $V_5(\Z_p)$, and $f^*$ is right $\GL_2(\Z_p)$-invariant.  Applying the Iwasawa decomposition, Lemma \ref{modulusChars}, Proposition \ref{alphachi}, and Lemma \ref{lem:lambdanonvan}, we find that
\[I(\Lambda,v_0,s) = I'(\Lambda,v_0,s) = \zeta_p(2s)\sum_{|t| \leq |\ell| \leq 1}{|t|^{-3}|\ell|^2 \lambda(\iota(t,\ell))|t|^s}.\]
Here the sum extends over all $t \in \mathbf{Q}_p^\times$ and $\ell \in L_p^\times$ satisfying $|t| \leq |\ell| \leq 1$ modulo the action of $\mathbf{Z}_p^\times \times \mathcal O_L^\times$.  It is this sum that we will analyze in the following two sections.  In fact, we will prove Theorem \ref{thm:unramcalc}, which says that the sum is $\lambda(1)L(\pi_p,\Spin,s)$.  Combining this with Proposition \ref{prop:ram} we obtain the following result.
\begin{theorem} \label{mainThm} Given a cusp form $\phi$ in the space of $\pi$, there exist $\phi_1$ also in the space of $\pi$, a section $f^*$ for the Eisenstein series, and a sufficiently large finite set $S$ of finite primes such that $I(\phi_1,s) = I_\infty(\phi,s)L^S(\pi,\Spin,s)$, where
\[I_\infty(\phi,s) = \int_{N(\R)Z(\R)\backslash H(\R)} f^*(g_1,s)\phi_{\chi}(g) dg.\]
Additionally, there exists $\alpha \in \mathcal{S}(V_5(\mathbf A))$ so that $I'(\phi_1,s) = I'_\infty(\phi,s)L^S(\pi,\Spin,s)$, where 
\[I_\infty'(\phi,s) = \int_{N_2(\R)U_{P,4}(\R)Z(\R)\backslash (\GL_2 \boxtimes \GSp_4)(\R)} f^{*}(g_1,s)\phi_{\chi}(g)\alpha_{\chi}(g_2) dg.\]
\end{theorem}
\section{Calculation with $\Delta$}\label{delta}
In this section and the next, we assume that $p \nmid 2D$ and that $\pi_p$ is unramified.  Denote by $\mathbf p$ the element $p \cdot 1_6$ of the center of $\GSp_6$.  When $p$ is split in $L$, pick $\pi_1, \pi_2$ two uniformizers above $p$ such that $p = \pi_1 \pi_2$.  We denote by $\tau$ the matrix
\[\left(\begin{array}{ccc|ccc} 
 1&  &  &  &  &  \\
 &  1&  &  &  &  \\
 &  &  p&  &  &  \\
\hline
 &  &  & p &  &  \\
 &  &  &  & 1 &  \\
 &  &  &  &  & p
\end{array}
\right).\]
The purpose of this section is to prove the following two theorems.
\begin{thm}\label{thm:deltainert}
Suppose that $p$ is inert in $L$.  Then
\[\int_{\GSp_6(\mathbf{Q}_p)} \lambda(g)\Delta(g)^s dg = \sum_{r \ge 0} \lambda(\iota(p^r,1))p^{6r-rs} - \sum_{r \ge 2} \lambda(\iota(p^{r-1},1)\tau)p^{6r-6-rs}.\]
\end{thm}
\begin{thm}\label{thm:deltasplit}
Suppose that $p$ is split in $L$.  Then
\begin{align*}\int_{\GSp_6(\mathbf{Q}_p)} \lambda(g)\Delta(g)^s dg =& \sum_{\substack{a \ge 0 \\ r \ge a \\ i \in \set{1,2}}} \lambda(\iota(p^r,\pi_i^a))p^{6r-2a-rs} - \sum_{r\ge 2} \lambda(\iota(p^{r-1},1)\tau)p^{6r-6-rs}\\
&-2 \sum_{r \ge 2} \lambda(\iota(p^{r-2},1)\mathbf{p})p^{6r-8-rs} \\&- \sum_{\substack{a \ge 2\\ r \ge a+1 \\ i \in \set{1,2}}} \lambda(\iota(p^{r-2},\pi_1^{a-1})\mathbf{p})p^{6r-2a-6-rs}.
\end{align*}
\end{thm}

Recall the definitions of $R, N_V, N_U,$ and $N_{U'}$ from Section \ref{subsec:groupdefs}.  Then
\begin{align}
	\nonumber &\int_{\GSp_6(\mathbf{Q}_p)} \lambda(g)\Delta^s(g) dg = \int_{R(\mathbf{Q}_p)} \delta_R^{-1}(g)\lambda(g)\Delta^s(g) dg\\
	\label{eqn:iterintegrals}=& \underbrace{\int_{M_R(\mathbf{Q}_p)}}_{\textrm{(D)}} \delta_R^{-1}(g)|\nu(g)|^s \lambda(g)\underbrace{\int_{N_V(\mathbf{Q}_p)}}_{\textrm{(C)}}\chi(n_v) \charf(n_vg) \underbrace{\int_{N_U(\mathbf{Q}_p)}}_{\textrm{(B)}} \chi(n_u)\\
	\nonumber&\hspace{6.5cm}\cdot\underbrace{\int_{N_{U'}(\mathbf{Q}_p)}}_{\textrm{(A)}} \charf(n_{u'}n_un_vg)\,du'\,du\,dv\,dg.
\end{align}
We evaluate these integrals from the inside out.

\subsection{Integral (A)}

\begin{defin}

Observe that any element of $\left(\GL_3(\mathbf Q_p) \cap M_3(\mathbf{Z}_p)\right)/\GL_3(\mathbf{Z}_p)$ may be represented uniquely by a matrix
\begin{equation} \label{eqn:mform} m = \(\begin{array}{ccc} p^a & \beta & \gamma_1 \\ 0 & p^b & \gamma_2 \\ 0 & 0 & p^c\end{array}\), \beta,\gamma_1 \in [0,p^a-1]\textrm{ and }\gamma_2 \in [0,p^b-1].\end{equation}
Such $m$ will arise as representatives of the right $\GL_3(\mathbf{Z}_p)$-equivalence class of the lower right hand corner of $g \in M_R(\mathbf{Q}_p)$.

We need only compute the inner integrals in (\ref{eqn:iterintegrals}) when $\lambda(g) \neq 0$.  By Lemma \ref{lem:lambdanonvan}, $\lambda(g) \neq 0$ implies that $a \geq c$ and $p^c|\beta$.  We define $\mathbf{M}$ to be the set of matrices $m$ of the form (\ref{eqn:mform}) meeting these two additional conditions.

For calculations, it will be more useful to also consider a larger, redundant set of representatives.  We define
\[\mathbf{M}'=\set{\left.\(\begin{array}{ccc} p^a & \beta & \gamma_1 \\ 0 & p^b & \gamma_2 \\ 0 & 0 & p^c\end{array}\)\right|\beta,\gamma_1,\gamma_2 \in \mathbf{Z}_p, a\ge c, p^c|\beta}.\]
There is a natural quotient map $\iota: \mathbf{M}' \rightarrow \mathbf{M}$ defined by replacing above-diagonal entries with representatives in (\ref{eqn:mform}).

We will also write $y$ for the minor $\mm{p^a}{\beta}{0}{p^b}$ of $m \in \mathbf{M}'$ or $\mathbf{M}$.

\end{defin}

We will be able to compute integrals (C) and (D) in (\ref{eqn:iterintegrals}) in terms of a sum over the entries of $\mathbf{M}$.

\begin{definition}

For $m \in \mathbf{M}'$, denote by $U_m$ the set of $u \in M_2(\mathbf Q_p)$ satisfying
\begin{itemize}
	\item ${}^tu = u$
	\item $uy \in M_2(\mathbf Z_p)$
	\item $(\gamma_1, \gamma_2)uy \in p^c\mathbf Z_p \oplus p^{\min(b,c)}\mathbf Z_p.$
\end{itemize}
Observe that $U_m$ is a subgroup of $M_2(\mathbf Q_p)$.
\end{definition}

\begin{remark}
One can check directly that if $m \in \mathbf{M}'$, $U_m$ depends only on its image $\iota(m)$ in $\mathbf{M}$.
\end{remark}

\begin{prop}\label{prop:uprimecond}
Suppose $n_v g \in M_6(\mathbf Z_p)$, and suppose the lower right $3 \times 3$ block of $n_v g$ is right $\GL_3(\mathbf{Z}_p)$-equivalent to $m \in \mathbf{M}'$.  Then
\begin{equation} \label{eqn:umeas} \int_{N_{U'}(\mathbf{Q}_p)} \charf(n_{u'}n_un_vg)\,du' = \begin{cases} p^{2c + \min{(b,c)}} &\mbox{if } u \in U_m \\ 0 & \mbox{otherwise}. \end{cases} \end{equation}
\end{prop}

\begin{proof}

Since the integrand and domain are right $\GL_3(\mathbf{Z}_p)$-invariant, we may assume the lower right $3 \times 3$ block of $n_vg$ is exactly $m$.  Putting our matrices in block form, we write $m = \left(\begin{array}{c|c}y & \gamma \\ \hline  & p^c\end{array}\right)$ and we write $\left(\begin{array}{c|c} r & u_3' \\ \hline u & \,^tr \end{array}\right)$ for the upper right $3 \times 3$ block of $n_{u'}n_{u}$.

A matrix multiplication shows that the integrality of $n_{u'}n_un_vg$ is equivalent to
\begin{equation} \label{eqn:intconds} u\cdot y \in M_2(\mathbf{Z}_p), u\gamma + p^c\,^tr \in \mathbf{Z}_p^2, ry \in \mathbf{Z}_p^2, \textrm{ and } r\gamma+p^cu_3' \in \mathbf{Z}_p.\end{equation}
Since $u_3'$ can always be chosen to satisfy the last condition, the existence of an element $u'$ making $\charf(n_{u'}n_un_vg) = 1$ is equivalent to $uy \in M_2(\mathbf{Z}_p)$, $ry \in \mathbf{Z}_p^2$, and $^t\gamma u + p^c r \in \mathbf{Z}_p^2$.  Multiplying this last condition by $y$, we find that integrality of $n_{u'}n_un_vg$ is equivalent to $uy \in M_2(\mathbf{Z}_p)$ and $\,^t\gamma uy$ in $\mathbf{Z}_p^2y + p^c\mathbf{Z}_p^2$.  Using the form of $y$ above, $\mathbf{Z}_p^2y + p^c\mathbf{Z}_p^2$ is equal to $p^c\mathbf Z_p \oplus p^{\min(b,c)}\mathbf Z_p.$  This proves that the integral is nonvanishing precisely when $u$ is in $U_m$.

Assuming that $u$ is in $U_m$, we now compute the integral $\int_{N_{U'}(\mathbf{Q}_p)} \charf(n_{u'}n_un_vg)\,du'$.   Since $u$ is in $U_m$, there is $u_0'$ in $N_{U'}$ so that $n_{u'_0}n_un_vg$ is integral.  Changing variables in the integral to shift by $u_0'$, we find that we must compute the measure of matrices $\left(\begin{array}{c|c} r & u_3' \\ \hline  & \,^tr \end{array}\right)$ so that the product
\[\left(\begin{array}{c|c} r & u_3' \\ \hline  & \,^tr \end{array}\right)\left(\begin{array}{c|c}y & \gamma \\ \hline  & p^c\end{array}\right)\]
is in $M_3(\mathbf{Z}_p)$.  The product is in $M_3(\mathbf{Z}_p)$ precisely when $ry$ and $p^c r$ in $\mathbf{Z}_p^2$ and $r\gamma + p^cu_3'$ in $\mathbf{Z}_p$.  Using the special form of $y$ above, the first condition occurs exactly when $r$ is in $p^{-c}\mathbf{Z}_p \oplus p^{-\min(b,c)}\mathbf Z_p$.  Thus the total measure of such matrices is $p^{2c + \min(b,c)}$, as claimed.

\end{proof}

\subsection{Integral (B)}

For each $m \in \mathbf{M}$, we must calculate
\[\int_{N_U} \chi(n_u) \charf_{U_m}(u) \,du.\]
As $U_m$ is a group, this is nonzero only when $U_m \subseteq \ker \chi$.

\begin{definition}
We say that an element $m' \in M_3(\mathbf{Z}_p)$ is called \emph{admissible} if it is right $\GL_3(\mathbf{Z}_p)$-equivalent to an element $m \in \mathbf{M}$ with $U_m \subseteq \ker \chi$.  Define the \emph{cumulative measure} $B(m')$ associated to an admissible $m'$ to be $p^{2c + \min{(b,c)}}\mathrm{meas}(U_m)$, and define $B(m')=0$ for any $m' \in M_3(\mathbf{Z}_p)$ that is not admissible.
\end{definition}

For the purpose of evaluating the iterated integral (\ref{eqn:iterintegrals}), we may replace the two innermost integrals by the cumulative measure $B(m')$.  Most of the effort in this section is devoted to finding the admissible $m \in \mathbf{M}$. The answer depends on whether $p$ is split or inert in $L$.  We will first prove some claims that apply in both cases.  For convenience, we will use the larger set $\mathbf{M}'$.

The proofs of these claims will follow from careful selection of a matrix $u$ to check that if certain conditions do not hold, we have $U_m \not\subseteq \ker\chi$.  It will be useful to write out the matrices in the defining conditions for $U_m$ explicitly: if $u = \mm{u_{11}}{u_{12}}{u_{12}}{u_{22}}$, then
\begin{align} \label{eqn:umcondexplicit}
	uy &= \mb{u_{11}p^a}{u_{11}\beta+u_{12}p^b}{u_{12}p^a}{u_{12}\beta+u_{22}p^b}\\
	\nonumber\textrm{ and }(\gamma_1,\gamma_2)uy &= (u_{11}p^a\gamma_1+u_{12}p^a\gamma_2, (u_{11}\beta+u_{12}p^b)\gamma_1 + (u_{12}\beta+u_{22}p^b)\gamma_2).
\end{align}

\begin{claim} \label{claim:cgeb}
If $m \in \mathbf{M}'$ is admissible, then $c \ge b$.
\end{claim}

\begin{proof}
If $b > c$, then $\mm{0}{0}{0}{p^{-1}} \in U_m$, so $U_m \not\subseteq \ker \chi$.
\end{proof}

\begin{claim} \label{claim:rowdivis}
If $m \in \mathbf{M}'$ is admissible, then one of $p^a, \beta,$ or $\gamma_1$ is not divisible by $p$.  Similarly, if $m$ is admissible, one of $p^b$ or $\gamma_2$ is not divisible by $p$.
\end{claim}

\begin{proof}
For the first case, we observe that if $a\ge 1$, $p|\beta$, and $p|\gamma_1$, then by the conditions $c \le a$ and $p^c|\beta$ together with Claim \ref{claim:cgeb}, we have $p^c|p^{a-1}\gamma_1$ and $p^b|p^{-1}\beta\gamma_1$, so $\mm{p^{-1}}{0}{0}{0} \in U_m$ and thus $U_m \not\subseteq \ker \chi$.

For the second case, we similarly find $\mm{0}{0}{0}{p^{-1}} \in U_m$, so $U_m \not\subseteq \ker \chi$.
\end{proof}

\begin{claim}
If $m \in \mathbf{M}'$ is admissible, then we have either $a = c$ or $b = c$.
\end{claim}

\begin{proof}
Assume that $a>c>b$.  Since $a \ge c+1 \ge 1$, $p|\beta$, and $c \ge b+1$, it follows that $\mm{p^{-1}}{0}{0}{0} \in U_m$.  So either $a = c$ or $c = b$.
\end{proof}

We may now state the classification of the admissible $m \in \mathbf{M}$ and corresponding measures $B(m)$.

\begin{remark} \label{remark:entrydiv}
Due to the right $\GL_3(\mathbf{Z}_p)$-action, for an element $m \in \mathbf{M}$, the matrix entries $\beta$ and $\gamma_1$ of $m$ can be thought of as elements of $\mathbf{Z}/p^a\mathbf{Z}$, and the entry $\gamma_2$ of $m$ can be thought of as an element of $\mathbf{Z}/p^b\mathbf{Z}$.  It makes sense, therefore, to write $\beta^{-1},\gamma_1^{-1},$ and $\gamma_2^{-1}$ for the inverses in these rings, taken to lie in $[0,p^a-1]$ or $[0,p^b-1]$, whenever these entries are invertible.  We will use this notation in the following classification result.
\end{remark}

\begin{theorem}[Inert case] \label{thm:inertclassify}
Suppose $p$ inert, and $m \in \mathbf{M}$ is admissible.  Then $m$ is of the form
\[\(\begin{array}{ccc} p^a & 0 & \gamma_1 \\  & 1 & 0 \\  &  & p^a\end{array}\),\]
where $p \nmid \gamma_1$ if $a \geq 1$.  The cumulative measure for such an $m$ is $B(m)=p^{2a}$.
\end{theorem}

\begin{theorem}[Split case] \label{thm:splitclassify}
Suppose $p$ is split.  There are several cases:
\begin{itemize}
\item ($b = 0, c = 0$.)  The admissible $m \in \mathbf{M}$ are then of the form
\[\(\begin{array}{ccc} p^a & \beta & \gamma_1 \\  & 1 & 0 \\  &  & 1\end{array}\),\]
where $\beta^2 \equiv D \pmod {p^a}$.  The cumulative measure for such an $m$ is $B(m)=p^a$.
\item ($b = 0, a = c, a \ge 1$.) The admissible $m \in \mathbf{M}$ are then of the form
\[\(\begin{array}{ccc} p^a & 0 & \gamma_1 \\  & 1 & 0 \\  &  & p^a\end{array}\),\]
where $p \nmid \gamma_1$.  The cumulative measure for such an $m$ is $B(m)=p^{2a}$.
\item ($a = c \ge b \ge 1$.)  The admissible $m \in \mathbf{M}$ are then of the form
\[\(\begin{array}{ccc} p^a & 0 & \gamma_1 \\  & p^b & \gamma_2 \\  &  & p^a\end{array}\),\]
where $(\gamma_1,p)= (\gamma_2,p) = 1$, and $(\gamma_1\cdot\gamma_2^{-1})^2 \equiv D \pmod {p^b}$.  The cumulative measure for such an $m$ is $B(m)=p^{2a+2b}$.
\item ($a > c =b \geq 1$.) Since in this case $p^c|\beta$, we may define $\beta' = p^{-c}\beta$.  The admissible $m \in \mathbf{M}$ are then of the form
\[\(\begin{array}{ccc} p^a & \beta & \gamma_1 \\  & p^c & \gamma_2 \\  &  & p^c\end{array}\),\]
where $p^c | \beta$, $(\gamma_1,p) = (\gamma_2,p) = 1$, $(\gamma_1\cdot\gamma_2^{-1})^2 \equiv D \pmod{p^c}$, $(\beta')^2 \equiv D \pmod{p^{a-c}}$, and $\beta' + (\gamma_1\cdot\gamma_2^{-1}) \equiv 0 \pmod{p^{\min{(a-c,c)}}}$.  The cumulative measure for such an $m$ is $B(m)=p^{a+3c}$.
\end{itemize}
\end{theorem}

\begin{remark}
The cumulative measure is $p^{a + 2b + c}$ in all cases.
\end{remark}

\begin{remark}
These conditions are all invariant under adding a multiple of $p^c$ to either $\gamma_1$ or $\gamma_2$.
\end{remark}

We now prove a succession of claims that will establish Theorems \ref{thm:inertclassify} and \ref{thm:splitclassify}.
\begin{claim}
Suppose that $m \in \mathbf{M}'$ has $b \ge 1$.  If $p$ is inert, $m$ is not admissible.  If $p$ is split and $m$ is admissible, then $(\gamma_1\cdot \gamma_2^{-1})^2 \equiv D \pmod {p^b}$.
\end{claim}

\begin{proof} \label{claim:bge1gammanecc}
If $b \ge 1$, then by Claim \ref{claim:cgeb}, $a \ge c \ge 1$, so by Claim \ref{claim:rowdivis} and Lemma \ref{lem:lambdanonvan}, $p$ cannot divide either $\gamma_1$ or $\gamma_2$, so $\gamma=\gamma_1\cdot \gamma_2^{-1}$ is in $\mathbf{Z}_p^\times$.  Observe that $u_0=p^{-b}\mm{-\gamma^{-1}}{1}{1}{-\gamma} \in U_m$.  Moreover, $\chi(u_0) = \psi(-p^{-b}\gamma^{-1}(\gamma^2-D))$.  To have $u_0\in\ker\chi$, we need $\gamma^2 \equiv D\pmod{p^b}$.  In the inert case, this is not possible.
\end{proof}

We will maintain this notation: for $m \in \mathbf{M}'$, if $b \ge 1$, we set $\gamma = \gamma_1\cdot\gamma_2^{-1} \in \mathbf{Z}_p^\times$.

\begin{claim} \label{claim:bge1gammasuff}
For $m \in \mathbf{M}'$, suppose that $b \geq 1$, $p$ splits in $L$, and $c = a$.  Then $(\gamma_1\cdot\gamma_2^{-1})^2 \equiv D \pmod{p^b}$ is a necessary and sufficient condition for $m$ to be admissible and, in this case, we have $B(m)=p^{2c + 2b}$.
\end{claim}

\begin{proof}
Let $u = \mm{u_{11}}{u_{12}}{u_{12}}{u_{22}}$, and suppose that $u \in U_m$.  By the integrality condition on $uy$, we have $p^au_{11}=p^cu_{11}$ integral, so $\beta u_{11}$ is also integral, and (looking at the upper-right entry of $uy$) $p^bu_{12}$ is integral.  So since $c \ge b$, $\beta u_{12}$ is integral, and (looking at the bottom-right entry of $uy$) $p^bu_{22}$ is integral.  By the condition on $(\gamma_1,\gamma_2)uy$ and the assumption $a=c$, we need $u_{11}p^a\gamma_1+u_{12}p^a\gamma_2 \in p^a \mathbf{Z}_p$, or $u_{11}\gamma+u_{12} \in \mathbf{Z}_p$.  In particular, since $p^bu_{12}$ is integral, $p^bu_{11}$ is also integral.  Therefore, $p^bu$ is integral.

We also need $(u_{11}\beta+u_{12}p^b)\gamma_1 + (u_{12}\beta+u_{22}p^b)\gamma_2$ in $p^b \mathbf{Z}_p$, or equivalently
\[(u_{11}\gamma+u_{12})\beta+(u_{12}\gamma + u_{22})p^b\]
in $p^b \mathbf{Z}_p.$ By the preceding calculation, this reduces to $u_{12}\gamma + u_{22} \in \mathbf{Z}_p$.  By setting $u_{12} = p^{-b}\alpha$ and solving for the other entries, we deduce that up to addition of a diagonal integral matrix, $u$ has the form $p^{-b}\alpha\mm{-\gamma^{-1}}{1}{1}{-\gamma}$ for $\alpha \in \mathbf{Z}_p$.  We checked in Claim \ref{claim:bge1gammanecc} that $\gamma^2 \equiv D \pmod{p^b}$ is a necessary condition.  When $\gamma$ satisfies this congruence, any $u$ of the form just calculated has $\chi(u)=1$, so $m$ is admissible under this condition.

Once the choice of $u_{12} \in p^{-b}\mathbf{Z}_p$ has been made, the choices of $u_{11}$ and $u_{22}$ have measure 1.  So using Proposition \ref{prop:uprimecond}, we obtain a cumulative measure of $p^{2c+2b}$.

\end{proof}

For $m \in \mathbf{M}'$, we set $\beta_0 = p^{-c}\beta \in \mathbf{Z}_p$.

\begin{claim}
For $m \in \mathbf{M}'$, assume that $a > c=b \ge 1$ and that $p$ is split in $L$.  Then if $m$ is admissible, we must have $(\beta_0,p) = 1$ and $\beta_0^2 \equiv D \pmod{p^{a-b}}$.
\end{claim}

\begin{proof}
If $p|\beta_0$, we observe that $\mm{p^{-1}}{0}{0}{0} \in U_m$, so we must have $(\beta_0,p)=1$.  For the second claim, observe that $u=p^{b-a}\mm{1}{-\beta_0}{-\beta_0}{\beta_0^2} \in U_m$ and $\chi(u) = \psi(\frac{\beta_0^2-D}{p^{a-b}})$.  If $m$ is admissible, we must have $\beta_0^2 \equiv D \pmod{p^{a-b}}$.
\end{proof}

\begin{claim}
For $m \in \mathbf{M}'$, assume that $a > c=b \ge 1$ and that $p$ is split in $L$.  Then if $m$ is admissible, we must have $-\beta_0(\gamma_1\cdot\gamma_2^{-1}) \equiv D \pmod {p^{\min(a-b,b)}}.$
\end{claim}

\begin{proof}
Define $b' = \min(a-b,b)$.  Observe that $u = p^{-b'}\mm{1}{0}{0}{-\beta_0\gamma} \in U_m$.  We must have $\chi(u) = \psi(p^{-b'}(-D-\beta_0\gamma)) = 1$, so $-\beta_0(\gamma_1\cdot\gamma_2^{-1}) \equiv D \pmod {p^{\min(a-b,b)}}$.
\end{proof}

\begin{claim}\label{claim:ageceb}
For $m \in \mathbf{M}'$, assume that $a > c=b \ge 1$ and that $p$ is split in $L$.  If $(\beta_0,p) = 1,$ $\beta_0^2 \equiv D \pmod{p^{a-b}},$ $\gamma^2 \equiv D \pmod{p^c},$ and $-\beta_0(\gamma_1\cdot\gamma_2^{-1}) \equiv D \pmod{p^{\min(a-b,b)}},$ then $m$ is admissible and $B(m) = p^{a+3c}$.
\end{claim}

\begin{proof}
If we add the congruences $\beta_0^2 \equiv D \pmod{p^{a-b}}$ and $-\beta_0(\gamma_1\cdot\gamma_2^{-1}) \equiv D \pmod{p^{\min(a-b,b)}}$, we obtain $\beta_0(\beta_0-\gamma) \equiv 2D \pmod{p^{\min(b,a-b)}}$.  Since $\beta_0$ and $2D$ are both $p$-adic units (the latter by hypothesis) and $\min(a-b,b)>1$, we deduce that $\beta_0-\gamma$ is also a $p$-adic unit.  We can also subtract the congruences to similarly obtain $\beta_0+\gamma \equiv 0 \pmod{p^{\min(a,b)}}$.

From here on, we will handle the cases $a -b \ge b$ and $b \ge a-b$ separately.

{\bf Case $a - b \ge b$}: We may write any symmetric $2\times 2$ matrix in the form
\[u=\mb{u_{11}}{-\beta_0u_{11}+u_{12}'}{-\beta_0u_{11}+u_{12}'}{\beta_0^2u_{11}-(\beta_0+\gamma)u_{12}' + u_{22}'}.\]
The integrality conditions become
\begin{align*}
uy = \mb{p^a u_{11}}{p^bu_{12}'}{-p^a\beta_0 u_{11}+p^au_{12}'}{-\gamma p^b u_{12}'+p^bu_{22}'} &\in M_2(\mathbf{Z}_p)\\ \textrm{ and }\gamma_2(p^a(\gamma-\beta_0)u_{11}+p^au_{12}',p^bu_{22}') &\in p^b\mathbf{Z}_p^2.
\end{align*}
These conditions imply $u_{12}' \in p^{-b}\mathbf{Z}_p$ and $u_{22}' \in \mathbf{Z}_p$.  The inequality $a-b \ge b$ implies $p^au_{12}' \in p^b\mathbf{Z}_p$.  It follows from the second integrality condition that $(\gamma-\beta_0)p^au_{11} \in p^b\mathbf{Z}_p$ as well.  We showed above that $\gamma-\beta_0 \in \mathbf{Z}_p^\times$, so $p^au_{11} \in p^b\mathbf{Z}_p$.

Conversely, if $u_{11} \in p^{b-a}\mathbf{Z}_p$, $u_{12}' \in p^{-b}\mathbf{Z}_p$, and $u_{22}' \in \mathbf{Z}_p$, the integrality conditions hold.  Moreover, we have
\[\chi(u) = \psi(-Du_{11}+\beta_0^2u_{11}-(\beta_0+\gamma)u_{12}' + u_{22}') = \psi((\beta_0^2-D)u_{11}-(\beta_0+\gamma)u_{12}')=1\]
since $\beta_0^2-D \in p^{a-b}\mathbf{Z}_p$ and $\beta_0+\gamma \in p^b \mathbf{Z}_p$.  It follows that the cumulative measure is $p^{a-b+b}\cdot p^{2c+\min(b,c)} = p^{a+3c}$.

{\bf Case $b \ge a-b$}: We may write any symmetric $2 \times 2$ matrix in the form
\[u = \mb{u_{11}}{-\gamma u_{11} + u_{12}'}{-\gamma u_{11} + u_{12}'}{\gamma^2 u_{11} - (\beta_0+\gamma)u_{12}' + u_{22}'}.\]
The integrality conditions become
\begin{align*}
	uy = \mb{p^au_{11}}{p^b(\beta_0-\gamma)u_{11}+p^bu_{12}'}{-p^a\gamma u_{11}+p^au_{12}'}{p^b\gamma(\gamma-\beta_0)u_{11}-p^b\gamma u_{12}'+p^bu_{22}'} &\in M_2(\mathbf{Z}_p)\\
	\textrm{ and }\gamma_2(p^au'_{12},p^bu_{22}') &\in p^b\mathbf{Z}_p^2.
\end{align*}
For these conditions to hold, we must have $u_{22}' \in \mathbf{Z}_p$ and $u_{12}' \in p^{b-a}\mathbf{Z}_p$.  In particular, since $b \ge a-b$, $p^bu_{12}' \in \mathbf{Z}_p$.  Since $\gamma-\beta_0 \in \mathbf{Z}_p^\times$, we obtain $p^bu_{11} \in \mathbf{Z}_p$ from the condition on $p^b(\beta_0-\gamma)u_{11}+p^bu_{12}'$.  Conversely, if $u_{22}' \in \mathbf{Z}_p$, $u_{12}' \in p^{b-a}\mathbf{Z}_p$, and $p^bu_{11} \in \mathbf{Z}_p$, the integrality conditions are met.

We calculate that under these conditions,
\[\chi(u) = \psi((\gamma^2-D)u_{11} - (\gamma+\beta_0)u_{12}' + u_{22}') = 1\]
since $u_{22}' \in \mathbf{Z}_p$, $p^{a-b}|\gamma+\beta_0$, and $p^b|\gamma^2-D$.  It follows that the cumulative measure is $p^{a-b+b}\cdot p^{2c+\min(b,c)} = p^{a+3c}$.

\end{proof}
This finishes the cases with $b \ge 1$.

\begin{claim} Assume that $b= c =0$ and that $m \in \mathbf{M}'$ is admissible.  If $p$ is inert, then $U_m \subseteq \ker \chi$ implies  $a = 0$ (so $a=c$).  If $p$ is split, $U_m \subseteq \ker \chi$ implies that $\beta^2 \equiv D \pmod{p^a}$.  These are sufficient conditions too, and we have $B(m)=p^a$ in both cases.
\end{claim}

\begin{proof}
First observe that if $b=c=0$, the conditions $uy \in \mathbf{M}_2(\mathbf{Z}_p)$ and ${}^tu=u$ are sufficient to verify that $u \in U_m$.  Observe that $p^{-a}\mm{1}{-\beta}{-\beta}{\beta^2} \in U_m$.  For $m$ to be admissible, we need $\chi(u) = \psi(p^{-a}(\beta^2-D))=1$ or $\beta^2 \equiv D \pmod{p^a}$.  If $p$ is inert, this forces $a=0$ as claimed, and the calculation of $B(m)$ is clear.

So assume $p$ is split and $\beta^2 \equiv D \pmod{p^a}$.  For $u = \mm{u_{11}}{u_{12}}{u_{12}}{u_{22}}$, $uy=\mm{p^au_{11}}{\beta u_{11}+u_{12}}{p^au_{12}}{\beta u_{12}+u_{22}}$.  The condition $u \in U_m$ is equivalent to $u_{11} \in p^{-a}\mathbf{Z}_p$, $u_{12} = -\beta u_{11}+ u_{12}'$ for some $u_{12}' \in \mathbf{Z}_p$, and $u_{22} = -\beta u_{12} + u_{22}'$ for some $u_{22}'\in\mathbf{Z}_p$.  For $u$ of this form,
\[\chi(u) = \psi(-Du_{11}+u_{12}) = \psi(u_{11}(\beta^2-D))=1.\]
The choice of $u_{11}$ gives a measure of $p^a$, the elements $u_{12}'$ and $u_{22}'$ are integral, and $b=c=0$, so the cumulative measure is $B(m)=p^a$.
\end{proof}

\begin{claim} Assume that $b = 0$ and $a = c > 0$.  Then $m \in \mathbf{M}'$ is admissible if and only if $p^a | \beta$ and $(p,\gamma_1) = 1$.  The cumulative measure in this case is $B(m)=p^{2a}$.
\end{claim}

\begin{proof}
If $m$ is admissible, then $p^a|\beta$ by definition of $\mathbf{M}'$ and $(pm\gamma_1)=1$ by Claim \ref{claim:rowdivis}.  Now suppose that these conditions hold and that $u = \mm{u_{11}}{u_{12}}{u_{12}}{u_{22}} \in U_m$.  Using $p^a|\beta$, the condition $uy\in M_2(\mathbf{Z}_p)$ simplifies to $u_{11} \in p^{-a}\mathbf{Z}_p$, $u_{12}\in\mathbf{Z}_p$, and $u_{22} \in \mathbf{Z}_p$.  Using $u_{12}\in\mathbf{Z}_p$, the condition on $(\gamma_1, \gamma_2)uy$ simplifies to $\gamma_1p^au_{11} \in p^a\mathbf{Z}_p$.  Since $(p,\gamma_1)=1$, we require $u_{11} \in \mathbf{Z}_p$.  It is clear that $U_m \subseteq \ker\chi$ and that the cumulative measure is $B(m)=p^{2a}$.
\end{proof}

This concludes the proof of Theorems \ref{thm:inertclassify} and \ref{thm:splitclassify}.

\subsection{Integral (C)}

The element $g \in M_R$ is determined by its lower right hand $3 \times 3$ minor and upper left hand entry $w$.  By right $K$-invariance, to calculate integral (C) in (\ref{eqn:iterintegrals}), we may assume that $w=p^r$ for some $r \in \mathbf{Z}_{\ge 0}$.  For $m \in \mathbf{M}'$ and $r \in \mathbf{Z}_{\ge 0}$, write $g_{m,r}$ for the element of $M_R$ determined by $m$ and $r$.  We must evaluate
\begin{equation} \label{eqn:intc} \int_{V(\mathbf Q_p)} \chi(v)\charf_{M_6(\mathbf{Z}_p)}(n_vg_{m,r})B(m) \,dv.\end{equation}
Observe that integrality of $n_vg_{m,r}$ is equivalent to integrality of $vx$ and $vz$.  If $r \ge a$ and $m$ is admissible, then $w{}^ty^{-1}$ is integral.  Hence, as $vx = vzw{}^{t}y^{-1}$, $vz$ integral implies $vx$ integral for admissible $m$.  Thus, the integral (\ref{eqn:intc}) just gives $B(m)$ when $r \geq a$.  In fact, the integral is $0$ when $r < a$.
\begin{prop}
Suppose that $m \in \mathbf{M}$ is admissible.  If $r < a$ then
\[\int_{V(\mathbf Q_p)} \chi(v)\charf_{M_6(\mathbf{Z}_p)}(n_vg_{m,r})B(m) \,dv = 0.\]
\end{prop}
\begin{proof}
Using the notation in (\ref{eqn:mform}), we calculate
\begin{align*} vx &= vzw{}^ty^{-1} = (p^{-c}\gamma_1,p^{-c}\gamma_2)\mb{p^{r+c-a}}{}{-\beta p^{r+c-a-b}}{p^{r+c-b}}\\ &=(p^{r-a}\gamma_1-p^{r-a-b}\beta\gamma_2,p^{r-b}\gamma_2).\end{align*}
In the inert case, we have $\beta=0$, so we require $p^{r-a}\gamma_1 \in \mathbf{Z}_p$.  However, if $a=0$ then $r < a$ is impossible, while if $a \ge 1$, $p\nmid \gamma_1$, so $p^{r-a}\gamma_1$ can never be integral if $r < a$.

We now handle the split cases; note that we combine the second and third cases of Theorem \ref{thm:splitclassify} below.  We write $v = (v_1,v_2)$ below.
\begin{itemize}
	\item ($b=0,c=0$.) We have $\gamma_2=0$, so the integrality condition simplifies to $p^{r-a}\gamma_1 \in \mathbf{Z}_p$.  Since $c=0$, we have $v=(\gamma_1,\gamma_2)$, so $\chi(v)\charf_{M_6(\mathbf{Z}_p)}(n_vg_{m,r})B(m) = \psi(v_1)B(m)$ on the domain where $v_1\in p^{a-r}\mathbf{Z}_p$ and $v_2 \in \mathbf{Z}_p$.  The desired vanishing follows.
	\item ($a=c \ge b$.) The argument is identical to the one in the inert case. 
	\item ($a > c = b \ge 1$.) Note that the conditions $p\nmid\gamma_2$ and $p^{r-b}\gamma_2 \in \mathbf{Z}_p$ force $b \le r < a$.  Write $\beta'=p^{-c}\beta$.  We require $p\nmid \gamma_1$ and $b=c$, so if $b\le r < a$, the integrality condition reduces to $p^{r-a}(\gamma_1 - \beta_0\gamma_2) \in \mathbf{Z}_p$.  Since $\gamma_2$ is a unit, this is equivalent to $\gamma-\beta_0 \in p^{a-r}\mathbf{Z}_p$.  Finally, using the condition $p\nmid 2D$, we showed in the proof of Claim \ref{claim:ageceb} that $\gamma-\beta_0$ is a unit, a contradiction if $r < a$.
\end{itemize}
\end{proof}

\subsection{Integral (D)}

As described earlier, integral (C) in (\ref{eqn:iterintegrals}) vanishes unless $r \geq a$.  Integral (D) is therefore the sum over $m \in \mathbf{M}$ and $r \ge a$ of $B(g_{m,r})\delta_R^{-1}(g_{m,r})|\nu(g_{m,r})|^s\lambda(g_{m,r})$.  Using Theorems \ref{thm:inertclassify} and \ref{thm:splitclassify} to enumerate the $m \in \mathbf{M}$, we deduce Theorems \ref{thm:deltainert} and \ref{thm:deltasplit}.  In the split case, one uses Lemma \ref{lem:hlem} to rewrite the result in the desired form.
\section{Calculation with $N$}\label{N}
Here we give the calculation of $N(s) * \zeta_p(2s-6)\sum{|t|^{s-6}|l|^2\lambda(\iota(t,l))}.$
\subsection{Hecke operators}
In the definition of $N$ above, $T_{0,3}$ is the Hecke operator that is the characteristic function of 
\[\GSp_6(\mathbf Z_p)\diag(1,1,1,p,p,p)\GSp_6(\mathbf Z_p),\]
$T_{3,3}$ is the Hecke operator that is the characteristic function of 
\[\GSp_6(\mathbf Z_p)\diag(p,p,p,p,p,p)\GSp_6(\mathbf Z_p),\]
and $T_{2,3}$ is the Hecke operator that is the characteristic function of 
\[\GSp_6(\mathbf Z_p)\diag(1,p,p,p,p,p^2)\GSp_6(\mathbf Z_p).\]

We first reduce these Hecke operators to $\GL_3$ Hecke operators.  To this end, suppose $m$ is an element of $\left(\GL_3(\mathbf Q_p) \cap M_3(\mathbf Z_p)\right)$, such that $p m^{-1} \in M_3(\mathbf Z_p)$. Consider the double coset
\[\GL_3(\mathbf Z_p)m\GL_3(\mathbf Z_p),\]
and suppose it has a coset decomposition
\[ \coprod_\beta{ v_\beta \GL_3(\mathbf Z_p)}.\] 
We use the shorthand $\lambda(g[m]_p)$ to mean
\[\sum_\beta {\lambda(g\tilde{v}_\beta)},\]
where $\tilde{v}_\beta$ is the element of the Levi of the Siegel parabolic with similitude $p$ whose top left $3 \times 3$ block is $v_\beta$. Write $M_P$ for the Levi of the Siegel parabolic.
\begin{proposition}\label{GL3T03} Suppose $g \in M_P$, $g = \iota(t,\ell)$, with $|t| \leq |\ell| \leq 1$.  Then 
\begin{align*}T_{0,3}\lambda(g) &= \lambda\left(g\left[\begin{array}{ccc} 1 & & \\ & 1 & \\ & & 1 \end{array}\right]_p\right) + p\lambda\left(g\left[\begin{array}{ccc} p & & \\ & 1 & \\ & & 1 \end{array}\right]_p\right) + p^3\lambda\left(g\left[\begin{array}{ccc} p & & \\ & p & \\ & & 1 \end{array}\right]_p\right) \\ &+ p^6\lambda\left(g\left[\begin{array}{ccc} p & & \\ & p & \\ & & p \end{array}\right]_p\right).\end{align*}
\end{proposition}
\begin{proof} The decomposition into cosets of Hecke operators for the symplectic group is standard; see, for example, Andrianov \cite[Lemma 3.49]{a2}.  However, the proof of the proposition is facilitated by a good choice of coset representatives, and its truth depends on the hypothesis that $g = \iota(t,\ell)$ with $|t| \leq |\ell| \leq 1$, so we briefly explain.

For an element $m \in M_P$ define
\[U(m) = \{u \in U_{P}(\mathbf Q_p) : mu \in M_6(\mathbf Z_p)\}.\]
Also, for ease of notation, denote by $X_0$ the set of representatives appearing in the sum for the operator $\left[1_3\right]_p$.  Similarly, denote by $X_1, X_2, X_3$ the sets of representatives appearing in the operators 
\[\left[\begin{array}{ccc} p & & \\ & 1 & \\ & & 1 \end{array}\right]_p,\; \left[\begin{array}{ccc} p & & \\ & p & \\ & & 1 \end{array}\right]_p,\; \text{ and }\left[\begin{array}{ccc} p & & \\ & p & \\ & & p \end{array}\right]_p,\]
respectively.

It follows from the result \cite[Lemma 3.49]{a2} mentioned above that for general $g$ one has
\[T_{0,3}\lambda(g) = \sum_{0 \leq i \leq 3}\sum_{x \in X_i}\sum_{\substack{u \in U(x) \\ \mod {U_P(\mathbf Z_p)}}}{{{\lambda(gxu)}}}.\]
We have
\[\lambda(gxu) = \chi(gxux^{-1}g^{-1})\lambda(gx).\]
One may check that for the $x$ appearing above, representatives $u \in U(x)/U_P(\mathbf Z_p)$ may be chosen so that $xux^{-1} \in M_6(\mathbf Z_p)$.  (This simplification is limited to the similitude $p$ case.)  Now, if $g = \iota(t,\ell)$, with $|t| \leq |\ell| \leq 1$, and $n \in U_R(\mathbf Z_p)$, then $\chi(gng^{-1}) = 1$.  Hence, to compute $T_{0,3}\lambda(g)$, one must just count the sizes of the sets $U(x)/U_P(\mathbf Z_p)$.  These sizes are $1, p, p^3,$ or $p^6$ if $x$ is in $X_0, X_1, X_2,$ or $X_3$, respectively.  The proposition follows.
\end{proof}

Set $\epsilon_L(p)$ to be $1$ if $p$ is split in $L$, and $-1$ if $p$ is inert.  Then we have the following decomposition of $T_{2,3}$.
\begin{proposition}\label{GL3T23} Suppose $g = \iota(t,\ell)$, with $|t| \leq |\ell| \leq 1$.  Then
\begin{align*} T_{2,3} \lambda(g) =&\lambda\left(g\left(\begin{array}{c|c} 1_3 & \\ \hline & p 1_3 \end{array}\right)\left[\begin{array}{ccc} 1 & & \\ & p & \\ & & p \end{array}\right]_p\right) + p^4\lambda\left(g\left(\begin{array}{c|c} p 1_3 & \\ \hline & 1_3 \end{array}\right)\left[\begin{array}{ccc} p & & \\ & 1 & \\ & & 1 \end{array}\right]_p\right)\\
&+ (p^3-1 + \delta_{|t|,|\ell|}(-p^3 + \epsilon_L(p)p^2))\lambda(\iota(t,\ell)).\end{align*}
Here $\delta_{|t|,|\ell|}$ is $1$ when $|t| = |\ell|$ and is $0$ otherwise.
\end{proposition}
\begin{proof} By Andrianov \cite[Lemma 3.49]{a2} and arguments similar to those in the proof of Proposition \ref{GL3T03}, if we set $U(p)'$ to be the elements $u$ of $U(p1_6)$ such that $(p1_6)u$ has rank one over $\mathbf F_p$, then
\begin{align*}
	T_{2,3}\lambda(g) =& \lambda\left(g\left(\begin{array}{c|c} 1_3 & \\ \hline & p 1_3 \end{array}\right)\left[\begin{array}{ccc} 1 & & \\ & p & \\ & & p \end{array}\right]_p\right) + p^4\lambda\left(g\left(\begin{array}{c|c} p 1_3 & \\ \hline & 1_3 \end{array}\right)\left[\begin{array}{ccc} p & & \\ & 1 & \\ & & 1 \end{array}\right]_p\right)\\&+\sum_{u \in U(p)'}{\lambda(gu)}.\end{align*}
The calculation of the term with coefficient $p^4$ uses that $g = \iota(t,\ell)$ with $|t| \leq |\ell| \leq 1$ together with integrality properties of a nice choice of representatives.  We have $\lambda(gu) = \chi(gug^{-1})\lambda(g)$ and
\[\sum_{u \in U(p)'}{\chi(gug^{-1})} = p^3 -1 + \sum_{u \in U(p)'}{\left(\chi(gug^{-1})-1\right)}.\]
If $|t| < |\ell|$, then $\chi(gug^{-1}) = 1$, and the sum vanishes.  If $|t| = |\ell|$, then $\chi(gug^{-1}) = \chi(u)$, and one evaluates the sum to be $-p^3 + \epsilon_L(p)p^2$.
\end{proof}

We will reduce the $\GL_3$ Hecke operators above to the following $\GL_2$ Hecke operators.
\begin{definition} Define
\[T(u) = \left(\begin{array}{c|c|c|c}1 & & & \\ \hline & u& & \\ \hline & & p{}^{t}u^{-1} & \\ \hline & & &p \end{array}\right).\]
For $g \in G$, we denote $(T_p \lambda)(g)$ to be the sum $\sum_u{ \lambda\left(gT(u)\right)}$ where $u$ in the sum varies over the set of size $p+1$ containing $\left(\begin{array}{cc} 1 & \\ & p\end{array}\right)$ and $\left(\begin{array}{cc} p & a\\ & 1\end{array}\right)$ for distinct representatives $a$ of the integers modulo $p$.  Similarly, we define
\[T'(u) = \left(\begin{array}{c|c|c|c}p & & & \\ \hline & u& & \\ \hline & & p{}^{t}u^{-1} & \\ \hline & & &1 \end{array}\right)\]
and denote $(T'_p\lambda)(g)$ to be the sum $\sum_u{ \lambda\left(gT'(u)\right)}$, where $u$ varies over the same set.
\end{definition}

We will relate $T_p$ and $T'_p$ shortly, and then reduce all the Hecke operators down to $T_p$.  To do this, we first need a lemma.
\begin{lemma}\label{identities} We have the equalities
\begin{itemize}
\item $\iota(t,\ell)\diag(1,1,1,p,p,p) = \iota(t/p,\ell) (p 1_6),$ 
\item $\iota(t,\ell)\diag(p,p,p,1,1,1) = \iota(pt,\ell),$
\item $\iota(t,\ell)\diag(p,1,1,p,p,1) = \iota(pt,p\ell),$ and
\item $\iota(t,\ell)\diag(1,p,p,1,1,p) = \iota(t/p,\ell/p)(p1_6).$ \end{itemize} \end{lemma}
Using these identities we can prove the following.
\begin{lemma}\label{T'T} $(T'_p\lambda)((t,\ell)) = (T_p\lambda)(\iota(p^2t,p\ell))$.
\end{lemma}
\begin{proof}
This follows from Lemma \ref{identities} and the equality
\[T'(u) = (p^{-1}1_6)\left(\begin{array}{c|c} p 1_3 & \\ \hline & 1_3 \end{array}\right)\left(\begin{array}{ccc|ccc} p&  &  &  &  &  \\  &  1&  &  &  &  \\ &  &  1&  &  &  \\ \hline &  &  & p &  &  \\ &  &  &  & p &  \\  &  &  &  &  & 1 \end{array} \right)T(u).\]
\end{proof}
Now we have the following simplification of the action of $T_{0,3}$.
\begin{proposition}\label{GL2T03}
We have
\begin{align*}
	T_{0,3} \lambda(\iota(t,\ell)) &=\lambda(\iota(t/p,\ell)) + p^3\lambda(\iota(pt,p\ell)) + p(T_p\lambda)(\iota(t,\ell))\\ & + p^4(T_p\lambda)(\iota(p^2t,p\ell))+ p^3\lambda(\iota(t/p,\ell/p)) + p^6\lambda(\iota(pt,\ell)).
\end{align*}
\end{proposition}
\begin{proof} We apply Lemmas \ref{identities} and \ref{T'T} to Proposition \ref{GL3T03} and use the well-known coset decompositions for Hecke operators on $\GL_3$.  For example,
\begin{align*}
	&\GL_3(\mathbf Z_p)\left(\begin{array}{ccc}p & & \\ &1& \\ & &1\end{array}\right)\GL_3(\mathbf Z_p) =\\ &\coprod_{a,b}{\left(\begin{array}{ccc}p & a& b\\ &1& \\ & &1\end{array}\right)\GL_3(\mathbf Z_p)} \coprod{} \coprod_{a}{\left(\begin{array}{ccc}1 & & \\ &p& a\\ & &1\end{array}\right)\GL_3(\mathbf Z_p)}
\coprod{\left(\begin{array}{ccc}1 & & \\ &1& \\ & &p\end{array}\right)\GL_3(\mathbf Z_p)}.
\end{align*}
Here $a$ and $b$ range over distinct representatives of the integers modulo $p$.

We have $\chi(gng^{-1}) = 1$ for $n \in U_R(\mathbf Z_p)$ and $g = \iota(t, \ell)$ with $|t| \leq |\ell| \leq 1$ as in the proof of Proposition \ref{GL3T03}.  Using this with Lemma \ref{identities}, the first term above leads to $p^2\lambda(\iota(pt,p\ell))$.  The second and third terms combine to give $T_p\lambda(\iota(t,\ell))$.

We handle the other terms in Proposition \ref{GL3T03} similarly.  The first and fourth terms are immediate; the third term (with the $p^3$ coefficient) gives a $T_p'$, which one converts to a $T_p$ via Lemma \ref{T'T}.
\end{proof}

\begin{proposition}\label{GL2T23}
We have
\begin{align*}
	T_{2,3} \lambda((t,\ell)) &=p(T_p\lambda)(\iota(pt,p\ell)) + \lambda(\iota(t/p^2,\ell/p)) + p^6\lambda(\iota(p^2t,p\ell)) \\ &+p^4(T_p\lambda)(\iota(pt,\ell)) + (p^3-1 + \delta_{|t|,|\ell|}(-p^3 + \epsilon_L(p)p^2))\lambda(\iota(t,\ell)).
\end{align*}
\end{proposition}
\begin{proof} This is completely analogous to Proposition \ref{GL2T03}.\end{proof}

\subsection{Some lemmas}
We will soon use Proposition \ref{GL2T03} and \ref{GL2T23} to compute $N(s)$ applied to the local Dirichlet series.  Before doing so, we need some lemmas concerning $\lambda$.
\begin{lemma} \label{TpVanish}
If $|\ell| = |p^{r-2}|$, then $(T_p \lambda)(\iota(p^{r-1},p\ell)) = 0$.
\end{lemma}
\begin{proof} Observe that
\[\frac{1}{p}\left(\begin{array}{cc} p&a \\ &1 \end{array}\right) \left(\begin{array}{cc} &1 \\ 1& \end{array}\right) \left(\begin{array}{cc} p& \\ a&1 \end{array}\right) = \left(\begin{array}{cc} 2a&1 \\ 1& \end{array}\right).\]
Set
\[u = \left(\begin{array}{cc} p&a \\ &1 \end{array}\right) \textrm{ and } w = \left(\begin{array}{cc} &1 \\ 1& \end{array}\right).\]
Since $\iota(p^{r-1},p\ell)n(w)\iota(p^{r-1},p\ell)^{-1} \equiv n(w/p) \mod N_L$, we have
\[\lambda(\iota(p^{r-1},p\ell)T(u)) = \lambda(\iota(p^{r-1},p\ell)T(u)n(w)) = \psi(-D\frac{2a}{p})\lambda(\iota(p^{r-1},p\ell)T(u)).\]
We conclude that $\lambda(\iota(p^{r-1},p\ell)T(u)) = 0$ when $a \neq 0$.  This argument handles $p-1$ of the $p+1$ terms in the sum defining $T_p$; the other two terms are similar.
\end{proof}

\begin{proposition}\label{TpExtremal1}
We have $(T_p \lambda)(\iota(p^r,1)) = \lambda(\iota(p^r,1)\tau)$.  Moreover, if $|\ell| = |p^r|$, then 
\[(T_p\lambda)(\iota(p^r,\ell)) = \begin{cases} 0 & p\ \mathrm{inert}\\ \lambda(\iota(p^{r-1},\ell\pi_1/p)) + \lambda(\iota(p^{r-1},\ell\pi_2/p)) & p\ \mathrm{ split.} \end{cases}\]
\end{proposition}
\begin{proof} If $u =\left(\begin{array}{cc} p&a \\ &1 \end{array}\right)$, then $\iota(p^r,1)T(u)$ does not satisfy the conditions of Lemma \ref{lem:lambdanonvan}: $p$ does not divide the top row $y$.  Thus only one term contributes to the sum defining $T_p\lambda(\iota(p^r,1))$, giving the first part of the proposition.

Now suppose $|\ell| = |p^r|.$  If $u = \left(\begin{array}{cc} 1& \\ &p \end{array}\right)$, then by taking $w = \left(\begin{array}{cc} 1& \\ & \end{array}\right)$, it follows as in Lemma \ref{TpVanish} that $\lambda(\iota(p^r,\ell)T(u)) = 0$.  Similarly, if $u = \left(\begin{array}{cc} p& a\\ &1 \end{array}\right)$, then by taking $w = \left(\begin{array}{cc} & \\ &1 \end{array}\right)$, it follows as in Lemma \ref{TpVanish} that $\lambda(\iota(p^r,\ell)T(u)) = \psi(\frac{-Da^2+1}{p})\lambda(\iota(p^r,\ell)T(u))$.

It follows that $\lambda(\iota(p^r,\ell)T(u)) = 0$ when $p$ is inert, giving the proposition in this case.  If $p$ is split, then the only two terms that may contribute to the sum defining $T_p\lambda(\iota(p^r,\ell))$ are the ones with $a$ a squareroot of $D^{-1}$ modulo $p$.  These two terms are right $M_R \cap \GSp_6(\mathbf Z_p)$-equivalent to the two terms $(p^{-1},\pi_i/p)(p1_6)$, $i = 1,2$, giving the proposition in the split case.
\end{proof}

\begin{proposition}\label{TpExtremal2} If $k \geq 1$ and $p$ is split, then
\[(T_p\lambda)(\iota(p^r,\pi_i^k)) = \lambda(\iota(p^{r-1},\pi_i^{k-1}))\]
for $i = 1,2$.
\end{proposition}
\begin{proof} Using Lemma \ref{lem:hlem}, one sees that $\iota(p^r,\pi_1^k)T(u)$ does not satisfy the nonvanishing criterion Lemma \ref{lem:lambdanonvan} unless $p{}^tu^{-1} = \left(\begin{array}{cc} p&a \\ &1 \end{array}\right)$ with $a \equiv h \mod p$.  Again from Lemma \ref{lem:hlem}, this $T(u)$ is right $M_R \cap \GSp_6(\mathbf Z_p)$-equivalent to $\iota(p^{-1},\pi_2/p)(p1_6)$.  This gives the lemma for $i=1$; the $i=2$ case is the same.
\end{proof}

\subsection{Final computation}
The results of the previous two subsections give us enough tools to compute $N(s)$ applied to the Dirichlet series
\[D(s) = \zeta_p(2s-6)\sum{|t|^{-6}|\ell|^2 \lambda(\iota(t,\ell))|t|^s}\]
directly.  The lengthy calculation is summarized in the following two theorems.
\begin{theorem}[Inert case] \label{thm:inert}
Suppose $p$ is inert in $L$.  The coefficient of $p^{-rs}$ in $N(s)D(s)$ is
\[\lambda(\iota(p^r,1))p^{6r} - \lambda(\iota(p^{r-1},1)\tau)p^{6r-6}.\] \end{theorem}

\begin{theorem}[Split case] \label{thm:split}
Suppose $p$ is split in $L$.  The coefficient of $p^{-rs}$ in $N(s)D(s)$ is
\[\lambda(\iota(p^r,1))p^{6r} - \lambda(\iota(p^{r-1},1)\tau)p^{6r-6}\]
\[+\sum_{i=1,2, 1 \leq k \leq r}{\lambda(\iota(p^r,\pi_i^k))|\pi_i^k|^2p^{6r}}
-\sum_{0 \leq j \leq r-2}{p^4 |\pi_i^j|^2\lambda(\iota(p^{r-2},\pi_i^j))p^{6r-12}}.\] \end{theorem}
\begin{proof} This is a direct, but tedious calculation.  Define $D'(s)$ to be the sum above in $D(s)$, except without the zeta factor, i.e., $\zeta_p(2s-6)D'(s)=D(s)$.  The first step is to use Propositions \ref{GL2T03} and \ref{GL2T23} to write an expression for the terms in $N(s)D'(s)$.  Expanding all the sums, there is a large amount of cancellation between the different terms in the expansion of $N(s)$.  Lemma \ref{TpVanish} and Propositions \ref{TpExtremal1} and \ref{TpExtremal2} give additional simplification.  Multiplying the result by the zeta factor $\zeta_p(2s-6)$ yields a telescoping expression, which gives the theorem.

Now we give the details.  First, we split up $N$ into two parts, $N_1$ and $N_2$, where $N_2$ contains the $\delta_{|t|,|\ell|}$ terms, and $N_1$ contains all the rest of the terms.  The coefficient of $p^{-rs}$ in $N_2 * D'(s)$ is
\begin{align*}&\sum_{|\ell| = |p^{r-2}|}{(-p^2)(-p^3+\epsilon(p)p^2)\lambda(\iota(p^{r-2}, \ell))|\ell|^2p^{6(r-2)}} \\&+ \sum_{|\ell| = |p^{r-4}|}{(-p^7)(-p^3+\epsilon(p)p^2)\lambda(\iota(p^{r-4}, \ell))|\ell|^2p^{6(r-4)}}. \end{align*}

Now we compute the coefficient of $p^{-rs}$ in $N_1 * D'(s)$.  We obtain
\begin{flalign*} 
 & \sum_{|p^r| \leq |\ell| \leq 1}{\lambda(\iota(p^r,\ell))|\ell|^2p^{6r}} & \\
+&p^{-2s}\cdot (-p^2) \left\{\sum_{|p^{r-2}| \leq |\ell| \leq 1}\left[\lambda(\iota(p^{r-4},\ell/p))+ p(T_p\lambda)(\iota(p^{r-1},p\ell)) + p^6\lambda(\iota(p^r, p \ell)) \right. \right. & \\ &\left. \left. +p^4(T_p\lambda)(\iota(p^{r-1},\ell)) + (p^4+p^3+p^2)\lambda(\iota(p^{r-2},\ell))\right]|\ell|^2 p^{6(r-2)}\vphantom{\sum_{|p^{r-2}| \leq |\ell| \leq 1}}\right\}& \\
+&p^{-3s} \cdot (p^5+p^4) \left\{\sum_{|p^{r-3}| \leq |\ell| \leq 1} \left[\lambda(\iota(p^{r-4},\ell))+p^3\lambda(\iota(p^{r-2},p\ell)) \right. \right. & \\ &\left. \left. + p(T_p\lambda)(\iota(p^{r-3},\ell)) + p^3\lambda(\iota(p^{r-4}, \ell/p))+p^4 (T_p\lambda)(\iota(p^{r-1},p\ell)) \right. \right. & \\ &\left. \left. + p^6\lambda(\iota(p^{r-2},\ell))\right]|\ell|^2p^{6(r-3)}\vphantom{\sum_{|p^{r-3}| \leq |\ell| \leq 1}}\right\}& \\
+&p^{-4s} \cdot(-p^7) \left\{\sum_{|p^{r-4}| \leq |\ell| \leq 1}\left[\lambda(\iota(p^{r-6},\ell/p))+ p(T_p\lambda)(\iota(p^{r-3},p\ell)) \right. \right. & \\ &\left. \left. + p^6\lambda(\iota(p^{r-2}, p \ell))  +p^4(T_p\lambda)(\iota(p^{r-3},\ell)) + (p^4+p^3+p^2)\lambda(\iota(p^{r-4},\ell))\right]|\ell|^2 p^{6(r-4)}\vphantom{\sum_{|p^{r-4}| \leq |\ell| \leq 1}}\right\}& \\
+&p^{-6s}p^{15}\left\{\sum_{|p^{r-6}| \leq |\ell| \leq 1}\left[\lambda(\iota(p^{r-6},\ell))\right]|\ell|^2p^{6(r-6)}\right\}.&
\end{flalign*}

Now we group the terms together by the power of $p$ appearing in the $t$ spot in $\lambda(\iota(t,\ell))$.  Listing only the terms of the form $\lambda(\iota(p^r,\ell))$, i.e.\ when the power is $r$, we obtain
\begin{flalign*}
&\left(\sum_{|p^r| \leq |\ell| \leq 1}{\lambda(\iota(p^r,\ell))|\ell|^2p^{6r}}\right) - \left(\sum_{|p^{r-2}| \leq |\ell| \leq 1}{\lambda(\iota(p^r,p\ell))|p\ell|^2p^{6r}}\right),& \end{flalign*}
which simplifies to
\begin{flalign*}\textbf{Case }t=p^r:\ &\sum_{|p^r| \leq |\ell| \leq 1, p \nmid \ell}{\lambda(\iota(p^r,\ell))|\ell|^2p^{6r}}.&\end{flalign*}

We next list the terms of the form $\lambda(\iota(p^{r-1},\ell))$.  These are
\begin{flalign*}
&\sum_{|p^{r-2}| \leq |\ell| \leq 1}{\left[(-p^3)(T_p\lambda)(\iota(p^{r-1},p\ell))+(-p^6)(T_p\lambda)(\iota(p^{r-1},\ell))\right]|\ell|^2 p^{6(r-2)}} &\\
& + \sum_{|p^{r-3}| \leq |\ell| \leq 1}(p^9+p^8)\left[(T_p\lambda)(\iota(p^{r-1},p\ell))\right]|\ell|^2p^{6(r-3)}&\\
& = \sum_{|p^{r-2}| = |\ell|}{(-p^3)|\ell|^2 p^{6(r-2)}(T_p\lambda)(\iota(p^{r-1},p\ell))}+ \sum_{|p^{r-3}| =|\ell|}{p^6|p\ell|^2(T_p\lambda)(\iota(p^{r-1},p\ell))p^{6(r-2)}} &\\
& + \sum_{|p^{r-2}| \leq |\ell| \leq 1, p \nmid \ell}{(-p^6)|\ell|^2(T_p\lambda)(\iota(p^{r-1},\ell))p^{6(r-2)}}.\end{flalign*}
By Lemma \ref{TpVanish}, the first term is zero.  By Proposition \ref{TpExtremal1}, the second term is zero when $p$ is inert, and simplifies when $p$ is split.  We obtain
\begin{flalign*}\textbf{Case }p\ \textbf{inert},\ t=p^{r-1}:\ &(-p^6)\lambda(\iota(p^{r-1},1)\tau)p^{6r-12} &\end{flalign*}
by Proposition \ref{TpExtremal1} and
\begin{flalign*}\textbf{Case }p\ \textbf{split},\ t=p^{r-1}:\ &(-p^6) \lambda(\iota(p^{r-1},1)\tau)p^{6r-12} &\\& + \sum_{|\ell| = |p^{r-3}|}{p^2|\ell|^2\left[\lambda(\iota(p^{r-2},\pi_1 \ell)) + \lambda(\iota(p^{r-2},\pi_2\ell))\right]p^{6r-12}} & \\ &+ \sum_{1 \leq k \leq r-2, i = 1,2}{(-p^6)p^{-2k}\lambda(\iota(p^{r-2},\pi_i^{k-1}))p^{6r-12}} & \end{flalign*}
using Proposition \ref{TpExtremal2}.  

The terms of the form $\lambda(\iota(p^{r-2},\ell))$ are
\begin{flalign*}
& -p^2(p^4+p^3+p^2)\left\{\sum_{|p^{r-2}| \leq |\ell| \leq 1}{\lambda(\iota(p^{r-2},\ell))|\ell|^2p^{6r-12}}\right\} &\\ &+ (p^5+p^4)\left\{\sum_{|p^{r-3}| \leq |\ell| \leq 1}{|\ell|^2p^{6r-18}\left[p^3\lambda(\iota(p^{r-2},p\ell))+p^6\lambda(\iota(p^{r-2},\ell))\right]}\right\} & \\ &-\left\{\sum_{|p^{r-4}| \leq |\ell| \leq 1}{p^{13}\lambda(\iota(p^{r-2},p\ell))|\ell|^2p^{6r-24}} \right\}.\end{flalign*}
This simplifies to
\begin{flalign*}
&\sum_{|\ell| = |p^{r-2}|}{-(p^5+p^4)\lambda(\iota(p^{r-2},\ell))|\ell|^2p^{6r-12}}-p^6\left\{\sum_{|p^{r-2}| \leq |\ell| \leq 1}{\lambda(\iota(p^{r-2},\ell))|\ell|^2p^{6r-12}}\right\}&\\ &+ \left\{\sum_{|p^{r-3}| = |\ell|}{p\lambda(\iota(p^{r-2},p\ell))|\ell|^2p^{6r-12}}\right\} + p^2\left\{\sum_{|p^{r-3}| \leq |\ell| \leq 1}{\lambda(\iota(p^{r-2},p\ell))|\ell|^2p^{6r-12}}\right\},\end{flalign*}
which finally gives
\begin{flalign*}\textbf{Case }t=p^{r-2}:\ & \left\{\sum_{|\ell|=|p^{r-2}|}{-(p^5+p^4)\lambda(\iota(p^{r-2},\ell))|\ell|^2p^{6r-12}}\right\}&\\&-p^6\left\{\sum_{|p^{r-2}| \leq |\ell| \leq 1, p \nmid \ell}{\lambda(\iota(p^{r-2},\ell))|\ell|^2p^{6r-12}}\right\}& \end{flalign*}
since the third term is zero by Lemma \ref{lem:lambdanonvan}.

The terms of the form $\lambda(\iota(p^{r-3},\ell))$ combine to give
\begin{flalign*}
& \sum_{|p^{r-3}|=|\ell|}{p^5(T_p\lambda)(\iota(p^{r-3},\ell))|\ell|^2p^{6r-18}} + \sum_{|p^{r-3}| \leq |\ell| \leq 1}{p^6(T_p\lambda)(\iota(p^{r-3},\ell))|\ell|^2p^{6r-18}}& \\&- \sum_{|p^{r-4}| \leq |\ell| \leq 1}{p^6(T_p\lambda)(\iota(p^{r-3},p\ell))|p\ell|^2p^{6r-18}}&\\=& \sum_{|p^{r-3}| \leq |\ell| \leq 1, p \nmid \ell}{p^6(T_p\lambda)(\iota(p^{r-3},\ell))|\ell|^2p^{6r-18}} + \sum_{|\ell| = |p^{r-3}|}{p^5(T_p\lambda)(\iota(p^{r-3},\ell))|\ell|^2p^{6r-18}}.&\end{flalign*}
Breaking into cases depending on whether $p$ is inert or split, we obtain
\begin{flalign*}\textbf{Case }p\ \textbf{inert},\ t=p^{r-3}:\ &p^6\lambda(\iota(p^{r-3},1)\tau)p^{6r-18}&\end{flalign*}
or
\begin{flalign*}
\textbf{Case }p\ \textbf{split},\ & t=p^{r-3}:\ p^6\lambda(\iota(p^{r-3},1)\tau)p^{6r-18}&\\&+\sum_{i=1,2, 1 \leq k \leq r-3}{p^{6-2k}\lambda(\iota(p^{r-4},\pi_i^{k-1}))p^{6r-18}} & \\ &+\sum_{|\ell| = |p^{r-3}|}{p^5 \left\{\lambda(\iota(p^{r-4},\ell \pi_1^{-1}))+\lambda(\iota(p^{r-4},\ell \pi_2^{-1}))\right\}|\ell|^2p^{6r-18}}.&\end{flalign*}

The terms of the form $\lambda(\iota(p^{r-4},\ell))$ are
\begin{flalign*}
&-\sum_{|\ell|=|p^{r-2}|}{p^2\lambda(\iota(p^{r-4},\ell/p))|\ell|^2p^{6r-12}} + \sum_{|p^{r-3}| \leq |\ell| \leq 1}{p\lambda(\iota(p^{r-4},\ell/p))|\ell|^2p^{6r-12}}&\\ &+ \sum_{|\ell| = |p^{r-3}|}{(p^{-1}+p^{-2})\lambda(\iota(p^{r-4},\ell))|\ell|^2p^{6r-12}}-\sum_{|p^{r-4}| \leq |\ell| \leq 1}{p\lambda(\iota(p^{r-4},\ell))|p\ell|^2p^{6r-12}},&\end{flalign*}
which simplifies to give
\begin{flalign*}\textbf{Case }t=p^{r-4}:\ &\sum_{|\ell|=|p^{r-4}|}{-(p^{-2}+p^{-3})\lambda(\iota(p^{r-4},\ell))|\ell|^2p^{6r-12}}.& \end{flalign*}

The terms with $p^{r-6}$ combine to give zero.

Now we combine the above final expressions with the terms from $N_2$, which yields some cancellation. When $p$ is inert, we obtain
\[\lambda(\iota(p^r,1))p^{6r}-\lambda(\iota(p^{r-1},1)\tau)p^{6r-6}-p^6\left(\lambda(\iota(p^{r-2},1))p^{6r-12}-\lambda(\iota(p^{r-3},1)\tau)p^{6r-18}\right).\]
Multiplying by the zeta factor $\zeta_p(2s-6)$ telescopes the terms to give
\[\lambda(\iota(p^r,1))p^{6r}-\lambda(\iota(p^{r-1},1)\tau)p^{6r-6}\]
as desired. 

When $p$ is split, we cancel terms to obtain
\begin{align*}&\lambda(\iota(p^r,1))p^{6r}-p^6\lambda(\iota(p^{r-2},1))p^{6r-12}+ \sum_{i=1,2, 1 \leq k \leq r}{\lambda(\iota(p^r,\pi_i^k))|\pi_i^k|^2p^{6r}} \\-& p^6\left(\sum_{i=1,2, 1 \leq k \leq r-2}{\lambda(\iota(p^{r-2},\pi_i^k))|\pi_i^k|^2p^{6r-12}}\right)-\lambda(\iota(p^{r-1},1)\tau)p^{6r-6}\\+&p^6\lambda(\iota(p^{r-3},1)\tau)p^{6r-18}+\sum_{0 \leq j \leq r-2}{(-p^4)|\pi_i^j|^2\lambda(\iota(p^{r-2},\pi_i^j))p^{6r-12}}\\ -&p^6\left(\sum_{0 \leq j \leq r-4}{(-p^4)|\pi_i^j|^2\lambda(\iota(p^{r-4},\pi_i^j))p^{6r-24}}\right).\end{align*}
Multiplying by the zeta factor telescopes the terms to give the desired expression, completing the proof.
\end{proof}

Theorem \ref{thm:unramcalc} follows from comparing Theorems \ref{thm:deltainert} and \ref{thm:deltasplit} with Theorems \ref{thm:inert} and \ref{thm:split}.

\section{Ramified integral}\label{spin:Ram}
In this section we show that the data for the integrals $I$ and $I'$ can be chosen to trivialize the integral at the bad finite places.  We prove this by modifying the argument of \S 12 of Gan-Gurevich \cite{gg}.  For the reader's convenience, we have written out the details.

Fix a rational prime $p$ and (by abuse of notation) write $L$ for the completion at $p$ of the quadratic extension $L/\mathbf{Q}$ used in the preceding sections.  We write $\sigma$ for the Galois automorphism on $L$ (which interchanges the factors if $L = \mathbf{Q}_p\times \mathbf{Q}_p$).
\begin{prop} \label{prop:ram} Given $v_0$ in the space of $\pi_p$, there exists $v$ in the space of $\pi_p$ and a local section $f_p$ for the Eisenstein series, both of which depend only on $v_0$, so that for all $(U_R,\chi)$-models $\ell$,
\[\int_{U_B(\mathbf{Q}_p)Z(\mathbf{Q}_p)\backslash (\GL_{2} \boxtimes \GL_{2,L}^*)(\mathbf{Q}_p)}{f(g,s)\ell(\pi(g) v)\,dg} = \ell(v_0).\]
Additionally, $\alpha_p$ in $\mathcal{S}(V_5)$ may be chosen so that
\[\int_{N_2(\mathbf{Q}_p)N_L(\mathbf{Q}_p)Z(\mathbf{Q}_p)\backslash (\GL_{2} \boxtimes \GSp_4)(\mathbf{Q}_p)}{f(g,s)\ell(\pi(g) v)\alpha(v_Dg)\,dg} = \ell(v_0).\]
\end{prop}
Here is a simple lemma that will be used in the proof.
\begin{lemma}\label{fourierInv} Let $M$ be $\mathbf{Q}_p$ or $L$, viewed as a locally compact abelian group.  Let $dz$ denote the Haar measure on $M$ giving $\mathcal{O}_M$ measure 1, and let $| \cdot|:M^\times \rightarrow \R_{> 0}^\times$ be defined by the property $d(mz) = |m|dz$.  Then
\[\int_M{\psi(vz)|p|^n\charf(p^n z \in \mathcal O_M)\,dz} = \begin{cases} 1 & \textrm{if }v \in p^n\mathcal O_M \\ 0 & \textrm{otherwise} \end{cases}\]
and
\[\int_M{\psi(vz)\psi(-z)|p|^n\charf(p^n z \in \mathcal O_M)\,dz} = \begin{cases} 1 & \textrm{if }v \in 1+ p^n\mathcal O_M \\ 0 & \textrm{otherwise}. \end{cases}\]
\end{lemma}
\begin{proof} The first statement is a straightforward calculation, and the second statement follows immediately from the first.\end{proof}
\begin{proof}[Proof of Proposition \ref{prop:ram}]
The statement for the second integral follows from the first.  Indeed, if we choose $\alpha_p\in \mathcal{S}(V_5)$ so that $\alpha_p(v_Dg)$ has support $\GL_{2,L}^*(\mathbf{Q}_p) K'$, where $K'$ is a sufficiently small open compact subgroup of $\GSp_4(\mathbf{Q}_p)$, we are reduced to proving the statement for the first integral.

Define the congruence subgroup $K_N = \(1 + p^NM_6(\Z_p)\) \cap \GSp_6(\Q_p)$.  There is a $K_{n_0}$ stabilizing $v_0$.  Define $\varphi_1 \in \mathcal C_c^\infty(\Q_p)$ by $\varphi_1(z) = \psi(-z)|p|^{n_0} \charf(p^{n_0} z \in \Z_p).$ Now set 
\[v_1 = \int_{\mathbf G_a(\Q_p)}{\varphi_1(z) \pi(u(z)) v_0 \,dz},\]
where
\[u(z) = \left(\begin{array}{c|cc|cc|c}1& & & & & \\ \hline &1& &0&0& \\& & 1&0&z &\\ \hline & & &1& & \\& & & &1 & \\ \hline & & & & &1\end{array}\right).\]
Then $v_1$ is stabilized by some congruence subgroup $K_n$.

For an element $\gamma \in L$, we define $\gamma_+ = (\gamma+\sigma \gamma)/2$ and $\gamma_- = (\gamma-\sigma\gamma)/2$.  Now set $\varphi^L_{2,n}(\gamma) = \psi_L(-\gamma)|p|^n_L \charf(p^n\gamma \in \mathcal O_L)$ and $\varphi^L_{3,n}(\gamma) = |p|^n_L \charf(p^n\gamma \in \mathcal O_L)$, where $\psi_L(\gamma) = \psi(\gamma_+)$.  We also write $\gamma_2$ for the column vector $^t(\gamma_+,\gamma_-)$ and $\gamma_3 = ^t(D^{-1}\gamma_+,\gamma_-)$.  Define
\[u_2(\gamma) = \left(\begin{array}{cccc} 1& -^t\gamma_2 & &  \\  & 1_2 &  &  \\  &   &  1_2 &\gamma_2 \\  &     &  &1 \end{array}\right)\textrm{ and } u_3(\gamma) = \left(\begin{array}{cccc} 1&  &  ^t\gamma_3 &  \\  & 1_2 &  & \gamma_3 \\  &   &  1_2 & \\  &     &  &1 \end{array}\right).\]
Now set
\[v_2 =\int_{\mathbf G_a(L)}{\varphi^L_{2,n}(\gamma) \pi(u_2(\gamma)) v_1 \,d\gamma} \textrm{ and }v_3 =\int_{\mathbf G_a(L)}{\varphi^L_{3,n}(\gamma) \pi(u_3(\gamma)) v_2 \,d\gamma}.\]
Suppose that
\[g = \left(\begin{array}{cccc} \nu& & &\\ & a& b& \\ & c&d & \\ & & &1 \end{array}\right),\]
where the four-by-four matrix $\mm{a}{b}{c}{d}$ lies in $\GL_{2,L}^*$. Then
\begin{align*} \ell(\pi(g)v_3) &= \int_{L}{\varphi^L_{3,n}(\gamma) \ell(\pi(gu_3(\gamma)) v_2) \,d\gamma} \\ &=\( \int_{L}{\varphi^L_{3,n}(\gamma)\chi(gu_3(\gamma)g^{-1})\,d\gamma}\) \ell(\pi(g)v_2) \\ &= \(\int_{ L}{\varphi^L_{3,n}(\gamma)\psi(c_{22}\gamma_+ + c_{12}\gamma_-)\,d\gamma}\) \ell(\pi(g)v_2) \\ &= \charf(c \in p^nM_2(\Z_p))\ell(\pi(g)v_2).\end{align*}
Here we have used Lemma \ref{fourierInv} and the fact the the elements of $\GL_{2,L}^*(\mathbf{Q}_p)$ are those matrices in $\GSp_4(\mathbf{Q}_p)$ that have the form
\[\left(\begin{array}{cc|cc}a_{11}&a_{12}&b_{11}&b_{12} \\ Da_{12}&a_{11}&b_{12}&Db_{11}\\ \hline Dc_{22}&c_{12}&d_{11}&Dd_{21} \\c_{12}& c_{22}& d_{21}&d_{11}\end{array}\right).\]
Similarly $\ell(\pi(g) v_2) = \charf(d \in 1 + p^n \mathcal O_L) \ell(\pi(g)v_1).$  Set $K_{n,L} = \GL_{2,L}^* \cap \(1+ p^nM_4(\Z_p)\)$ and define $B_{L,n} = \set{\mm{b_1}{*}{}{b_2} \in \GL_{2,L}^*(\mathbf{Q}_p) : b_2 \in 1 + p^n \mathcal O_L}$.  Note the conditions $c \in p^n \mathcal O_L$ and $d \in 1 + p^n\mathcal O_L$ are together equivalent to
$\mm{a}{b}{c}{d} \in B_{L,n}K_{n,L}$.  Now there is $N \geq n$ so that $v_3$ is stable by $K_N$.  Set $K_{N,2} = \GL_2(\Q_p) \cap \(1 + p^NM_2(\Z_p)\)$.  Pick a nonzero section $f_N^*$ supported on $B_2K_{N,2}$, for example
\[f_N^*(g,s) = \int_{B_2}{\delta_{B_2}^{-s}(b)\charf(bg \in K_{N,2}) \,db}\]
where $db$ is a right Haar measure.  Then
\[\int_{U_BZ\backslash \GL_{2} \boxtimes \GL_{2,L}^*}{f(g_1,s)\ell(\pi(g) v_3)\,dg} =\int_{U_BZ\backslash B_2K_{N,2} \boxtimes \GL_{2,L}^*}{f(g_1,s)\ell(\pi(g) v_3)\,dg}\]
and this is, up to positive constants,
\begin{align*}
&= \int_{U_BZ\backslash B_2 \boxtimes \GL_{2,L}^*}{\delta_{B_2}^{-1}(g_1)f(g_1,s)\ell(\pi(g) v_3)\,dg}\\
&= \int_{N_L \backslash \GL_{2,L}^*}{|\nu(g_2)|^{s-1}\ell(\pi(g_2) v_3)\,dg}\end{align*}
where we have embedded $\GL_{2,L}^*$ in $\GSp_6$ via the elements
\[g_2 = \left(\begin{array}{cccc} \nu& & &\\ & a& b& \\ & c&d & \\ & & &1 \end{array}\right)\]
with $\mm{a}{b}{c}{d} \in \GL_{2,L}^*$.  Set $T_{L,n} = T_L \cap B_{L,n}$.  Then again up to a positive constant this is
\begin{align*} &= \int_{N_L \backslash B_{L,n}K_{n,L}}{|\nu(g_2)|^{s-1}\ell(\pi(g_2)v_1)\,dg} \\ &= \int_{T_{L,n}}{\delta_{B_L}^{-1}(t)|\nu(t)|^{s-1}\ell(\pi(t)v_1)\,dt}\end{align*}
Now, applying Lemma \ref{fourierInv} again as above, 
\[\ell(\pi(t) v_1) = \charf\(\frac{\nu}{N(\ell)} \in 1 +p^{n_0}\Z_p\) \ell(\pi(t)v_0)\]
where, in block diagonals, $t = \diag(\nu,\nu\,^t\ell^{-1},\ell,1)$.  Under the conditions $\frac{\nu}{N(\ell)} \in 1 +p^{n_0}\Z_p$, the support of this last integral becomes open compact, and the integrand becomes the constant $\ell(v_0)$ where the integrand is supported.  Hence, for some positive constant $C$, 
\[\int_{U_BZ\backslash \GL_{2} \boxtimes \GL_{2,L}^*}{f(g,s)\ell(\pi(g) v_3)\,dg} = C\ell(v_0).\]
Setting $v = C^{-1}v_3$ gives the proposition.
\end{proof}

\bibliography{integralRepnBib_final}

\begin{thebibliography}{10}

\bibitem{a3}
A.~N. Andrianov.
\newblock Shimura's hypothesis for {S}iegel's modular group of genus {$3$}.
\newblock {\em Dokl. Akad. Nauk SSSR}, 177:755--758, 1967.

\bibitem{a1}
A.~N. Andrianov.
\newblock Multiplicative arithmetic of {S}iegel's modular forms.
\newblock {\em Uspekhi Mat. Nauk}, 34(1(205)):67--135, 1979.

\bibitem{a2}
A.~N. Andrianov.
\newblock {\em Introduction to {S}iegel modular forms and {D}irichlet series}.
\newblock Universitext. Springer, New York, 2009.

\bibitem{bfg}
D.~Bump, M.~Furusawa, and D.~Ginzburg.
\newblock Non-unique models in the {R}ankin-{S}elberg method.
\newblock {\em J. Reine Angew. Math.}, 468:77--111, 1995.

\bibitem{bg}
D.~Bump and D.~Ginzburg.
\newblock Spin {$L$}-functions on symplectic groups.
\newblock {\em Internat. Math. Res. Notices}, (8):153--160, 1992.

\bibitem{gan}
W.~T. Gan.
\newblock An automorphic theta module for quaternionic exceptional groups.
\newblock {\em Canad. J. Math.}, 52(4):737--756, 2000.

\bibitem{gg}
W.~T. Gan and N.~Gurevich.
\newblock C{AP} representations of {$G_2$} and the spin {$L$}-function of
  {${\rm PGSp}_6$}.
\newblock {\em Israel J. Math.}, 170:1--52, 2009.

\bibitem{ginzburg}
D.~Ginzburg.
\newblock A construction of {CAP} representations in classical groups.
\newblock {\em Int. Math. Res. Not.}, (20):1123--1140, 2003.

\bibitem{gj}
D.~Ginzburg and D.~Jiang.
\newblock Periods and liftings: from {$G_2$} to {$C_3$}.
\newblock {\em Israel J. Math.}, 123:29--59, 2001.

\bibitem{grs2}
D.~Ginzburg, S.~Rallis, and D.~Soudry.
\newblock A tower of theta correspondences for {$G_2$}.
\newblock {\em Duke Math. J.}, 88(3):537--624, 1997.

\bibitem{grs}
D.~Ginzburg, S.~Rallis, and D.~Soudry.
\newblock On {F}ourier coefficients of automorphic forms of symplectic groups.
\newblock {\em Manuscripta Math.}, 111(1):1--16, 2003.

\bibitem{grossSavin}
B.~H. Gross and G.~Savin.
\newblock Motives with {G}alois group of type {$G_2$}: an exceptional
  theta-correspondence.
\newblock {\em Compositio Math.}, 114(2):153--217, 1998.

\bibitem{gs}
N.~Gurevich and A.~Segal.
\newblock {The {R}ankin-{S}elberg integral with a non-unique model for the
  standard $\mathcal{L}$-function of $\mathit{G}_2$}.
\newblock {\em J. Inst. Math. Jussieu}, 14:149--184, 1 2015.

\bibitem{langlandsEP}
R.~P. Langlands.
\newblock {\em Euler products}.
\newblock Yale University Press, New Haven, Conn.-London, 1971.
\newblock Yale Mathematical Monographs, 1.

\bibitem{llz}
A.~Lei, D.~Loeffler, and S.~L. Zerbes.
\newblock Euler systems for {R}ankin-{S}elberg convolutions of modular forms.
\newblock {\em Ann. of Math. (2)}, 180(2):653--771, 2014.

\bibitem{panchishkinVankov}
A.~Panchishkin and K.~Vankov.
\newblock Explicit {S}himura's conjecture for {${\rm Sp}_3$} on a computer.
\newblock {\em Math. Res. Lett.}, 14(2):173--187, 2007.

\bibitem{psr}
I.~Piatetski-Shapiro and S.~Rallis.
\newblock A new way to get {E}uler products.
\newblock {\em J. Reine Angew. Math.}, 392:110--124, 1988.

\bibitem{pollackShahKS}
A.~{Pollack} and S.~{Shah}.
\newblock {On the Rankin-Selberg integral of Kohnen and Skoruppa}.
\newblock {\em Math. Res. Lett.}, 24(1):173--222, 2017.

\bibitem{shahidi}
F.~Shahidi.
\newblock On the {R}amanujan conjecture and finiteness of poles for certain
  {$L$}-functions.
\newblock {\em Ann. of Math. (2)}, 127(3):547--584, 1988.

\bibitem{vo}
S.~C. Vo.
\newblock The spin {$L$}-function on the symplectic group {${\rm GSp}(6)$}.
\newblock {\em Israel J. Math.}, 101:1--71, 1997.

\end{thebibliography}
\bibliographystyle{abbrv}

\end{document}